\documentclass[12pt,a4paper,oneside]{amsart}
\usepackage{geometry}
\usepackage[utf8]{inputenc}
\usepackage{amssymb}
\usepackage{bbold}
\usepackage[usenames,dvipsnames]{xcolor}
\usepackage[inline]{enumitem}
\usepackage{mathtools}
\usepackage{lmodern}
\usepackage{xifthen}
\usepackage{stmaryrd} 
\usepackage{tikz-cd}

\definecolor{linkblue}{RGB}{1,1,190}
\definecolor{citegreen}{RGB}{1,190,1}
\usepackage[linkcolor=linkblue
          ,citecolor=citegreen
          ,ocgcolorlinks        
          ,bookmarksopen=True   
          ,pdfencoding=auto     
           ]{hyperref}

\usepackage{amsthm}

\usepackage[capitalize          
          ,noabbrev            
          ]{cleveref}

\usepackage{microtype}  

\theoremstyle{plain}
\newtheorem{theorem}{Theorem}[section]
\newtheorem{corollary}[theorem]{Corollary}
\newtheorem{lemma}[theorem]{Lemma}
\newtheorem{proposition}[theorem]{Proposition}

\theoremstyle{definition}
\newtheorem{definition}[theorem]{Definition}
\newtheorem{remark}[theorem]{Remark}

\newcommand{\subref}[1]{\dosubref\cref#1\relax}
\newcommand{\Subref}[1]{\dosubref\Cref#1\relax}
\def\dosubref#1#2:#3\relax{#1{#2}\ref{#2:#3}}

\crefname{pluralequation}{Equations}{Equations}
\Crefname{pluralequation}{Equations}{Equations}

\numberwithin{equation}{section} 

\setlist[enumerate,1]{label=\textup{(\arabic*)}, ref=\textup{(}\arabic*\textup{)}, leftmargin=0.75cm}
\setlist[enumerate,2]{label=\textup{(\roman*)}, ref=\textup{(}\roman*\textup{)}, leftmargin=0.5cm}
\setlist[enumerate,3]{label=\textup{(\Alph*)}, ref=\textup{(}\Alph*\textup{)}, leftmargin=0.5cm}

\setlist[itemize,1]{leftmargin=0.75cm}

\newlist{equivenumerate}{enumerate}{1}
\setlist[equivenumerate,1]{label=\textup{(\alph*)},ref=\textup{(\alph*)},leftmargin=0.75cm}

\newcommand{\bN}{\mathbb N}
\newcommand{\bQ}{\mathbb Q}
\newcommand{\bR}{\mathbb R}
\newcommand{\bZ}{\mathbb Z}
\newcommand{\cA}{\mathcal A}
\newcommand{\cR}{\mathcal R}
\newcommand{\cU}{\mathcal U}

\newcommand{\fp}{\mathfrak p}

\renewcommand{\sc}{\mathsf c}   
\newcommand{\sd}{\mathsf d}
\newcommand{\sv}{\mathsf v}

\newcommand{\sD}{\mathsf D}
\newcommand{\sL}{\mathsf L}
\newcommand{\sZ}{\mathsf Z}

\newcommand{\val}{\sv}
\newcommand{\fmon}[1]{\mathcal F^*({#1})}

\DeclareMathOperator{\spec}{spec}

\DeclareMathOperator{\adj}{adj}
\DeclareMathOperator{\chr}{char}
\DeclareMathOperator{\nr}{nr}
\DeclareMathOperator{\tr}{tr}
\DeclareMathOperator{\rad}{rad}

\DeclarePairedDelimiter{\length}{\lvert}{\rvert}
\DeclarePairedDelimiter{\abs}{\lvert}{\rvert}
\DeclarePairedDelimiter{\card}{\lvert}{\rvert}

\newcommand{\quat}[1]{{\mathsf{#1}}} 
\newcommand{\res}{\mathbb k}     

\def\rfop{*}
\makeatletter
\newcommand\rigidfactorization[2][]{%
  \def\rf@delim{\rfop}
  \newif\ifrf@notfirst
  #1
  \@for\next:=#2\do{%
    \ifrf@notfirst
      \rf@delim
    \fi
    \rf@notfirsttrue
    \next
  }%
}
\makeatother
\newcommand\rf\rigidfactorization

\newenvironment{claim}[1][]{%
\begin{description}[leftmargin=0pt]%
\ifthenelse{\isempty{#1}}{\item[Claim]}{\item[Claim {#1}]}
}%
{\end{description}}

\begin{document}

\title{Arithmetical invariants of local quaternion orders}

\author[N.~Baeth]{Nicholas R.~Baeth}
\author[D.~Smertnig]{Daniel Smertnig}

\address{Department of Mathematics\\
         Franklin \& Marshall College \\
         Lancaster, PA, 17604, USA}
\email{nicholas.baeth@fandm.edu}
\address{University of Graz\\
         NAWI Graz\\
         Institute for Mathematics and Scientific Computing\\
         Heinrichstra\ss e 36\\
         8010 Graz, Austria}
\email{daniel.smertnig@uni-graz.at}
\thanks{Much of this work was completed when the first author was employed at the University of Central Missouri. The second author was supported by the Austrian Science Fund (FWF) project P26036-N26.}

\keywords{elasticity, factorization theory, quaternion orders}

\subjclass[2010]{16H10, 11R27, 11S45}

\begin{abstract}
  Let $D$ be a DVR, let $K$ be its quotient field, and let $R$ be a $D$-order in a quaternion algebra $A$ over $K$.
  The \emph{elasticity} of $R^\bullet$ is $\rho(R^\bullet) = \sup\{\, k/l : u_1\cdots u_k = v_1 \cdots v_l \text{ with $u_i$, $v_j$ atoms of  $R^\bullet$ and $k$, $l \ge 1$} \,\}$ and is one of the basic arithmetical invariants that is studied in factorization theory.
  We characterize finiteness of $\rho(R^\bullet)$ and show that the set of distances $\Delta(R^\bullet)$ and all catenary degrees $\sc_\sd(R^\bullet)$ are finite.
  In the setting of noncommutative orders in central simple algebras, such results have only been understood for hereditary orders and for a few individual examples.
\end{abstract}

\maketitle

\section{Introduction}

If $R$ is a Noetherian ring, every non-zero-divisor $a \in R^\bullet$ can be written as a (finite) product of atoms (irreducible elements).
In general, such a factorization is not unique, and arithmetical invariants are used to describe this non-uniqueness.
If $\sL(a) = \{\, k \in \bN_0 : a = u_1\cdots u_k \text{ with } u_1, \ldots\,, u_k \text{ atoms} \,\}$ is the \emph{set of lengths} of $a$, then $\rho(a) = \sup \sL(a)/\min \sL(a) \in \bQ_{\ge 1} \cup \{\infty\}$ is the \emph{elasticity} of $a$ (where we set $\rho(a)=1$ for $a \in R^\times$).
The elasticity of $R^\bullet$ is then $\rho(R^\bullet) = \sup\{\, \rho(a) : a \in R^\bullet \,\} \in \bR_{\ge 1} \cup \{\infty\}$.
The study of arithmetical invariants, such as sets of lengths and elasticities, is part of factorization theory, a field which has been well-developed in the commutative setting (see \cite{Anderson,Chapman,GHK06,Fontana-Houston-Lucas,Chapman-Fontana-Geroldinger-Olberding,Geroldinger16} for recent monographs and surveys) and which has been recently extended to noncommutative settings (see \cite{BBG,Baeth-Smertnig,Smertnig16b,Smertnig17,Baeth-Jeffries,Bell-Heinle-Levandovskyy} for recent results).

In \emph{transfer Krull monoids of finite type} (as defined in \cite{Geroldinger16}), arithmetical invariants are always finite and can be expressed in terms of certain combinatorial invariants of an (abelian) class group associated to $R^\bullet$.
This holds in particular when $R$ is a ring of algebraic integers in a number field, where the class group is the usual one.
More recently, these results have been extended to a noncommutative setting: If $R$ is a hereditary order in a central simple algebra over a number field, then $R^\bullet$ is a transfer Krull monoid as well (see \cite{Smertnig17} or \cite{Smertnig,Baeth-Smertnig} for the special case of maximal orders).
Here, the class group is isomorphic to a ray class group of the center.
The proofs of these results proceed via multiplicative ideal theory, respectively, a structure theory for finitely generated projective modules over such rings.

In the commutative setting, non-hereditary (or, what is in this case equivalent, non-maximal) orders have been studied as well.
In particular, if $R$ is an order in a number field and $\overline R$ is its integral closure, then $\rho(R^\bullet) < \infty$ if and only if the map $\spec(\overline R) \to \spec(R)$ given by $\fp \mapsto \fp \cap R$ is bijective.
This can be proved by passing to localizations (using the notion of \emph{defining systems}), and studying the local situation, where $R_\fp^\bullet$ (for $\fp \in \spec(R)$) turns out to be a \emph{finitely primary monoid}.
(See \cite[Theorem 3.7.1]{GHK06} or see \cite{Kainrath} for a vast generalization.)

As orders in number fields have provided the nucleus for the development of the commutative theory, their natural noncommutative analogues, orders in central simple algebras over number fields, provide a good starting point and benchmark for the development of a noncommutative theory.
Since hereditary orders have been dealt with, the next step is to consider non-hereditary orders.

In this paper, we consider the simplest of these cases: Let $D$ be a discrete valuation ring, let $K$ be its quotient field, let $A$ be a quaternion algebra over $K$, and let $R$ be a $D$-order in $A$ (in particular, we restrict to the local case).

The main result of the present paper is as follows.
The definitions of the arithmetical invariants can be found in \cref{s-background}.
For a discussion of the Structure Theorem for Unions of Sets of Lengths and the definition of almost arithmetical progressions (AAPs) see \cite[Theorem 2.6 and Definition 2.5]{Geroldinger16}.
See also \cite{Fan-Geroldinger-Kainrath-Tringali}, where Fan, Geroldinger, Kainrath, and Tringali characterize when the Structure Theorem for Unions of Sets of Lengths holds.

\begin{theorem} \label{t-main}
  Let $D$ be a discrete valuation ring, let $K$ be its quotient field, and let $A$ be a central simple $K$-algebra.
  Let $R$ be a $D$-order in $A$.
  Suppose that the completion $\widehat A\cong A \otimes_K \widehat K$ is either a division ring or that $\widehat A \cong M_2(\widehat K)$.
  If $R$ is hereditary, then $\rho(R^\bullet)=1$, $\Delta(R^\bullet)=\emptyset$, and $\sc_\sd(R^\bullet) \le 2$ for any distance $\sd$ on $R^\bullet$.
  If $R$ is not hereditary, then
  \begin{enumerate}
  \item\label{t-main:rho} $\rho(R^\bullet) < \infty$ if and only if $\widehat A$ is a division ring.
  \item\label{t-main:delta} $\card{\Delta(R^\bullet)} < \infty$ and $\sc_\sd(R^\bullet) < \infty$ for any distance $\sd$ on $R^\bullet$.
  \item\label{t-main:structure} The Structure Theorem for Unions of Sets of Lengths holds for $R^\bullet$.
    Thus, if $\Delta(R^{\bullet}) \ne \emptyset$, there exist constants $k^*$ and $M^* \in \bN$ such that for all $k \ge k^*$, the union of sets of lengths $\cU_k(R^\bullet)$ is an AAP with difference $\min\Delta(R^\bullet)$ and bound $M^*$.
    \textup{(}If $\Delta(R^\bullet) = \emptyset$, then $\cU_k(R^\bullet)=\{k\}$.\textup{)}
  \end{enumerate}
  If $R$ is an Eichler order of level $n\ge 2$, then $\sc_\sd(R^\bullet) \le n+6$, $\min\Delta(R^\bullet)=1$, and $\max \Delta(R^\bullet)\le n+4$.
  In particular, $\cU_k(R^\bullet)=\bN_{\ge 2}$ for $k \ge 2$.
\end{theorem}

In this paper we deal with non-hereditary orders, as hereditary orders have been treated before.
If $R$ is a hereditary $D$-order in a central simple algebra (where $D$ is a DVR), then $R$ is a semilocal hereditary Noetherian prime ring.
In this case stable isomorphism of finitely generated projective $R$-modules implies isomorphism, and $R$ has trivial ideal class group (see \cite[\S40]{LevyRobson11}).
As a consequence of \cite[Theorem 4.4]{Smertnig17}, we then have $\rho(R^\bullet) = 1$, $\Delta(R^\bullet)=\emptyset$, and $\sc_\sd(R^\bullet) \le 2$.
Moreover, more precise information about unique factorization is available in the form of \cite[Proposition 4.12]{Smertnig17}.
(In this case, $\rho(R^\bullet)=1$ and $\Delta(R^\bullet)=\emptyset$ can also be derived using \cite{Estes}.)

Together with the results on hereditary orders, \cref{t-main} completely characterizes the finiteness of the stated arithmetical invariants for orders in quaternion algebras ($\dim_K A=4$) over the quotient field of a DVR.
In particular, if $R$ is such an order, then $\rho(R^\bullet) < \infty$ if and only if $R$ is hereditary or $\widehat A$ is a division ring.

\begin{remark}\label{r-largemindelta}
We have noted in \cref{t-main} that $\min\Delta(R^\bullet)=1$ if $R$ is an Eichler order. In fact, the authors know of no examples where $\min\Delta(R^\bullet)>1$, even when $R$ is not Eichler. It would be interesting to find such an example, if one exists.
\end{remark}

In the proof we distinguish three separate cases: Where the completion $\widehat A$ is a division ring, where $\widehat A \cong M_2(\widehat K)$ and $\widehat R$ is an Eichler order, and where $\widehat A \cong M_2(\widehat K)$ and $\widehat R$ is a non-Eichler order.
The first case is the simplest, running essentially parallel to the commutative one for finitely primary monoids of rank $1$ (in fact, here it is not necessary to require $\dim_K A =4$).
The second case is dealt with by explicit computations (here we also have the most explicit bounds on the invariants and are able to provide a complete classification of the atoms).
In the third and final case we make use of a parametrization of quaternion orders by means of ternary quadratic forms (a special case of \cite{Voight11}).

The paper is structured accordingly: After recalling some necessary background, we study the three cases as outlined above in \cref{s-division,s-eichler,s-noneichler}.
In \cref{s-eichler,s-noneichler} we assume at first that $D$ is complete.
In the final section, \cref{s-wrapup}, we tie everything together, culminating in the proof of \cref{t-main}.
The passage to the non-complete case is dealt with by constructing a transfer homomorphism to the completion.
This also allows us to apply the conclusions of \cref{s-eichler,s-noneichler} to the case when $D$ is not complete.

We thank the referee for a careful reading of a previous version of this manuscript and, in particular, for pointing out the questions in \cref{r-largemindelta} and \cref{r-finitelyprimarymonoid}.

\section{Background}
\label{s-background}

By $\bN_0$ we denote the set of nonnegative integers, and by $\bN$ the set of positive integers.
If $a$,~$b \in \bZ$ we denote by $[a,b]=\{\, x \in \bZ : a \le x \le b \,\}$ a discrete interval.
The symbol $\subset$ denotes an inclusion of sets that is not necessarily proper.
By a \emph{monoid} we mean a cancellative semigroup with identity.
If $H$ is a multiplicative semigroup with identity, we denote by $H^\bullet$ its monoid of cancellative elements, and by $H^\times$ its unit group.
A discrete valuation ring (DVR) is a commutative principal ideal domain with a unique non-zero prime ideal.
If $D$ is a DVR, we will always write $\pi \in D$ for a prime element.

\subsection{Factorizations and arithmetical invariants}

Let $H$ be a monoid.
An element $u \in H$ is an \emph{atom} if $u=ab$ with $a$,~$b \in H$ implies $a \in H^\times$ or $b \in H^\times$.
The set of atoms of $H$ is denoted by $\cA(H)$.
The monoid $H$ is \emph{atomic} if every non-unit can be expressed as a product of finitely many atoms.
Two elements $a$,~$b \in H$ are \emph{associated} if there exist $\varepsilon$, $\delta \in H^\times$ such that $a = \varepsilon b \delta$, and \emph{right associated} [\emph{left associated}] if $a=b\varepsilon$ [$a=\varepsilon b$].

Denote by $\fmon{\cA(H)}$ the free monoid with basis $\cA(H)$.
On the cartesian product $H^\times \times \fmon{\cA(H)}$ we define the following operation:
If $(\varepsilon,y)$,~$(\varepsilon',y') \in H^\times \times \fmon{\cA(H)}$ with $y=u_1\cdots u_k$ and $y'=v_1\cdots v_l$ where $u_1$, $\ldots\,$, $u_k$, $v_1$, $\ldots\,$,~$v_l \in \cA(H)$, then
\[
(\varepsilon,y)(\varepsilon',y') =
\begin{cases} (\varepsilon,u_1\cdots (u_k\varepsilon')v_1\cdots v_l) &\text{if $k >0$,} \\
              (\varepsilon\varepsilon', v_1\cdots v_l)               &\text{if $k=0$}.
\end{cases}
\]
With this product, $H^\times \times \fmon{\cA(H)}$ is a monoid.
On $H^\times \times \fmon{\cA(H)}$, we define a congruence $\sim$ by $(\varepsilon,y)\sim(\varepsilon',y')$ if all of the following hold:
\begin{itemize}
  \item $k=l$,
  \item $\varepsilon u_1\cdots u_k=\varepsilon' v_1\cdots v_k$ as product in $H$,
  \item either $k=0$, or there exist $\delta_2$, $\ldots\,$,~$\delta_k \in H^\times$ and $\delta_{k+1}=1$ such that
    \[
    \varepsilon' v_1 = \varepsilon u_1 \delta_2^{-1} \quad\text{and}\quad v_i = \delta_i u_i \delta_{i+1}^{-1} \quad\text{for all $i \in [2,k]$.}
    \]
\end{itemize}
\begin{definition}
  Let $H$ be a monoid.
  The quotient $\sZ^*(H) = H^\times \times \fmon{\cA(H)}/\sim$ is the \emph{monoid of \textup{(}rigid\textup{)} factorizations of $H$}.
  The class of $(\varepsilon,u_1\cdots u_k)$ in $\sZ^*(H)$ is denoted by $\rf[\varepsilon]{u_1,\cdots,u_k}$.
  The symbol $\rfop$ also denotes the operation on $\sZ^*(H)$.
  There is a natural homomorphism
  \[
  \pi=\pi_H \colon \sZ^*(H) \to H, \quad \rf[\varepsilon]{u_1,\cdots,u_k} \mapsto \varepsilon u_1\cdots u_k.
  \]
  For $a \in H$, the set $\sZ^*(a)=\sZ_H^*(a) = \pi^{-1}(a)$ is the set of \emph{\textup{(}rigid\textup{)} factorizations of $a$}.
  If $z=\rf[\varepsilon]{u_1,\cdots,u_k}$, then $\length{z}=k$ is the \emph{length} of $z$.
\end{definition}

By construction, $\rf[\varepsilon]{u_1,\cdots,u_{i},u_{i+1},\cdots,u_k}=\rf[\varepsilon]{u_1,\cdots,u_i\delta^{-1},\delta u_{i+1},\cdots,u_k}$ for all $\delta \in H^\times$ and for all $i \in [1,k-1]$.
Similarly, $\rf[\varepsilon]{u_1,\cdots,u_k} = \rf[1]{(\varepsilon u_1),\cdots,u_k}$ if $k \ge 1$.
In particular, as long as $k \ge 1$ (equivalently, $\pi(z) \not\in H^\times$), we may represent $z$ as $z=\rf{u_1',\cdots,u_k'}$ with atoms $u_1'$, $\ldots\,$,~$u_k'$ and omit the unit at the beginning.

We now define a number of arithmetical invariants, based on the lengths of factorizations alone.
More background on these arithmetical invariants can be found in the recent survey \cite{Geroldinger16}.

\begin{definition}
  Let $H$ be an atomic monoid and let $a \in H$.
  \begin{enumerate}
  \item $\sL(a)=\sL_H(a)=\{\, \length{z} : z \in \sZ_H^*(a) \,\} \subset \bN_0$ is the \emph{set of lengths} of $a$.
  \item $H$ is \emph{half-factorial} if $\card{\sL(a)}=1$ for all $a \in H$.
  \item A natural number $d \in \bN$ is a \emph{distance} of $a$ if there exist $k$,~$l \in \sL(a)$ with $l-k=d$ and $\sL(a)\cap [k,l]=\{k,l\}$.
    By $\Delta(a) \subset \bN$ we denote the \emph{set of distances} of $a$, and by
    \[
      \Delta(H) = \bigcup_{a \in H} \Delta(a)
    \]
    the \emph{set of distances} of $H$.
  \item For $k \in \bN$ we define
    \[
      \cU_k(H) = \bigcup_{\substack{a \in H\\ k \in \sL(a)}} \sL(a),
    \]
    the \emph{union of sets of lengths} containing $k$.
  \item For $k \in \bN$ the \emph{$k$-th elasticity} is $\rho_k(H) = \sup\, \cU_k(H) \in \bN \cup \{\infty\}$.
  \item The \emph{elasticity} of $a$ is
    \[
      \rho(a) = \frac{\sup \sL(a)}{\min \sL(a)} \in \bQ_{\ge 1} \cup \{\infty\}
    \]
    if $a \not \in H^\times$, and $\rho(a)=1$ if $a \in H^\times$.
    The \emph{elasticity} of $H$ is then $\rho(H) = \sup\{\, \rho(a) : a \in H \,\} \in \bR_{\ge 1} \cup \{\infty\}$.
  \end{enumerate}
\end{definition}

With these definitions, we have
\[
\rho(H) = \sup\Big\{\, \frac{\rho_k(H)}{k} \,:\,  k \in \bN \,\Big\} = \lim_{k\to \infty} \frac{\rho_k(H)}{k}.
\]

An \emph{almost arithmetical progression} (AAP) with difference $d\in \mathbb N$ and bound $M\in \mathbb N_0$ is a subset $L\subset \bN_0$ with $L\subset \min L+d\mathbb Z$ and
\[
  L\cap [\min L + M, \sup L-M]= (\min L + d \bZ) \cap [\min L + M, \sup L - M].
\]
We say that the \emph{Structure Theorem for Unions of Sets of Lengths} holds for $H$ if there are constants $d$, $k^\ast$ and $M$ such that for all $k\geq k^\ast$ every $\cU_k(H)$ is an AAP with distance $d$ and bound $M$. Again, we refer to \cite{Geroldinger16} for additional background.

Often it is easier to study sets of lengths in an alternate monoid and then pull back arithmetical information via an appropriate monoid homomorphism. A \emph{transfer homomorphism} $\varphi\colon H \to T$ of two monoids $H$, $T$ is a homomorphism satisfying:
\begin{itemize}
\item $T = T^\times \varphi(H) T^\times$ and $\varphi^{-1}(T^\times) = H^\times$.
\item If $a \in H$, $s$, $t \in T$ and $\varphi(a)=st$, then there exist $b$,~$c \in H$ and $\varepsilon \in T^\times$ such that $a=bc$, $\varphi(a) = s\varepsilon^{-1}$, and $\varphi(b) = \varepsilon t$.
\end{itemize}
A transfer homomorphism $\varphi$ is \emph{isoatomic} if, whenever $u$,~$v \in \cA(H)$ are such that $\varphi(u)$ and $\varphi(v)$ are (two-sided) associated in $T$, then already $u$ and $v$ are associated in $H$.
Transfer homomorphisms allow one to lift many factorization theoretical properties from $T$ to $H$.
In particular, $\sL_H(a) = \sL_T(\varphi(a))$ for all $a \in H$.
See \cite{Baeth-Smertnig}.

To introduce a more refined invariant, namely the catenary degree, we first need the notion of a distance between two factorizations.
(Note that the term `distance' in the following definition is unrelated to the one for the set of distances $\Delta(H)$.)

\begin{definition} \label{d-distance}
  Let $H$ be a monoid.
  Let $D = \{\, (z,z') \in \sZ^*(H) \times \sZ^*(H) : \pi(z)=\pi(z') \,\}$.
  A \emph{distance on $H$} is a map $\sd\colon D \to \bN_0$ satisfying each of the following properties for all $z$,~$z'$,~$z'' \in \sZ^*(H)$ with $\pi(z)=\pi(z')=\pi(z'')$ and all $x \in \sZ^*(H)$:
  \begin{enumerate}[label=\textup{(\textbf{D\arabic*})},ref=\textup{(D\arabic*)},leftmargin=*]
    \item\label{d:ref} $\sd(z,z) = 0$.
    \item\label{d:sym} $\sd(z,z') = \sd(z',z)$.
    \item\label{d:tri} $\sd(z,z') \le \sd(z,z'') + \sd(z'',z')$.
    \item\label{d:mul} $\sd(x\rfop z, x \rfop z') = \sd(z,z') = \sd(z \rfop x, z' \rfop x)$.
    \item\label{d:len} $\abs[\big]{\length{z} - \length{z'}} \le \sd(z,z') \le \max\left\{ \length{z}, \length{z'}, 1 \right\}$.
  \end{enumerate}
\end{definition}

This notion of a distance was introduced in \cite[\S3]{Baeth-Smertnig}, where several examples of distances can be found.
For any distance we can now define a corresponding catenary degree.

\begin{definition}[Catenary degree]
  Let $H$ be an atomic monoid and let $\sd$ be a distance on $H$.
  \begin{enumerate}
    \item
      Let $a \in H$ and $z$,~$z' \in \sZ^*(a)$.
      A finite sequence of rigid factorizations $z_0$, $\ldots\,$,~$z_n$ of $a$ is called an \emph{$N$-chain} (in distance $\sd$) between $z$ and $z'$ if
      \[
      z=z_0, \quad \sd(z_{i-1},z_i) \le N\text{ for all }i\in [1,n], \quad \text{and }z_n=z'.
      \]

    \item The \emph{catenary degree \textup{(}in distance $\sd$\textup{)} of $a$}, denoted by $\sc_\sd(a)$, is the minimal $N \in \bN_0 \cup \{\infty\}$ such that for any two factorizations $z$,~$z' \in \sZ^*(a)$, there exists an $N$-chain between $z$ and $z'$.

    \item The \emph{catenary degree \textup{(}in distance $\sd$\textup{)} of $H$} is
      \[
      \sc_\sd(H) = \sup\{\, \sc_\sd(a) : a \in H \,\} \in \bN_0 \cup \{\infty\}.
      \]
  \end{enumerate}
\end{definition}

\subsection{Orders in central simple algebras}

We refer to \cite{Reiner} for background on orders in central simple algebras.
Let $D$ be a DVR with quotient field $K$ and let $A$ be a central simple $K$-algebra.
The algebra $A$ is a \emph{quaternion algebra} if $\dim_K A = 4$.
A subring $R \subset A$ is a \emph{\textup{(}$D$-\textup{)}order} in $A$ if $D \subset R$, the $D$-module $R$ is finitely generated, and $KR=A$.

We denote by $\nr\colon A \to K$ the reduced norm and by $\tr\colon A \to K$ the reduced trace.
Every $x \in A$ is a root of its reduced characteristic polynomial,
\begin{equation} \label{e-redchar}
X^n - \tr(x)X^{n-1} + a_{n-2} X^{n-2} + \dots + a_1 X + (-1)^n \nr(x) \in K[X].
\end{equation}
If $x \in R$, then $x$ is integral over $D$, and hence the coefficients of the reduced characteristic polynomial are all in $D$.

We recall that some basic multiplicative properties of an element $x$ of $R$ may be characterized in terms of $\nr(x)$.

\begin{lemma} \label{l-basic}
  Let $x \in R$.
  \begin{enumerate}
    \item\label{l-basic:unit} $x \in R^\times$ if and only if $\nr(x) \in D^\times$.
    \item\label{l-basic:zd} $x \in R^\bullet$ if and only if $\nr(x) \ne 0$.
    \item\label{l-basic:atom} If $\nr(x) \in \cA(R^\bullet)$ then $x \in \cA(R^\bullet)$.
  \end{enumerate}
\end{lemma}

\begin{proof}
  \ref*{l-basic:unit}
  If $xy =1$ with $y \in R$, then $\nr(x)\nr(y)=\nr(xy)=1$, and hence $\nr(x) \in D^\times$.
  If $\nr(x) \in D^\times$, then \cref{e-redchar} implies $x \in R^\times$.

  \ref*{l-basic:zd}
  As in \ref*{l-basic:unit}, one shows $x \in A^\times$ if and only if $\nr(x) \ne 0$.
  Thus, if $\nr(x) \ne 0$ then $x \in R^\bullet$.
  If $x \in R^\bullet$, then $x$ is also cancellative in $A$.
  Since $A^\bullet=A^\times$, we have $x\in A^\times$ and so $\nr(x) \ne 0$.

  \ref*{l-basic:atom}
  Suppose $x$ is not an atom.
  Then $x=yz$ with $y$,~$z \in R^\bullet \setminus R^\times$.
  But then $\nr(x) =\nr(y)\nr(z)$ with $\nr(y)$,~$\nr(z) \in D^\bullet \setminus D^\times$ by \ref*{l-basic:unit}.
  Thus $\nr(x)$ is not an atom.
\end{proof}

We now note two simple facts that hold for all orders that are also local rings. These observations will be used in later sections.
By $J(R)$ we denote the Jacobson radical of $R$.

\begin{lemma} \label{l-pipower}
  Suppose $R$ is a local ring \textup{(}that is, $R/J(R)$ is a division ring\textup{)}.
  \begin{enumerate}
  \item \label{l-pipower:1}
    $\cA(R^\bullet) = J(R)^\bullet \setminus \{\, ab : a, b \in J(R) \,\}$.
    In particular, every element in $J(R)^\bullet \setminus J(R)^2$ is an atom.

  \item \label{l-pipower:2}
    There exists an $N \in \bN_0$ such that any product of $N$ atoms of $R^\bullet$ is divisible by $\pi$.
    In particular, every $x \in R^\bullet$ can be represented in the form $x=\pi^m \varepsilon u_1\cdots u_n$ with $\varepsilon \in R^\times$, $m \in \bN_0$, $u_1$, $\ldots\,$,~$u_n \in \cA(R^\bullet)$, and $n < N$.
  \end{enumerate}
\end{lemma}

\begin{proof}
  \ref*{l-pipower:1}
  This is immediate from the definitions, since $J(R) = R \setminus R^\times$.

  \ref*{l-pipower:2}
  Since $\overline R = R/\pi R$ is a finite dimensional $D/\pi D$-algebra, it is Artinian.
  Thus there exists $N \in \bN_0$ such that $J(\overline R)^N = \mathbf 0$.
  This implies $J(R)^N \subset \pi R$.
  Since every product of $N$ or more atoms is contained in $J(R)^N$, the result follows.
\end{proof}

\paragraph{\textbf{Quaternion orders.}}
Suppose now that $A$ is a quaternion algebra, that is $\dim_K A = 4$.
Then there exists a standard involution $\overline{\mkern2mu\cdot\mkern2mu}\colon A \to A$, the \emph{conjugation}.
Thus, $\overline{\mkern2mu\cdot\mkern2mu}$ is a $K$-linear anti-automorphism of $A$ with $\overline{\overline{x}}=x$ and $x \overline x \in K$ for all $x \in A$. In particular, we have
\[
\tr(x) = x + \overline{x} \quad\text{and}\quad \nr(x) = x \overline{x} = \overline{x} x \quad\text{for $x \in A$}.
\]
If $x \in R$, then $\overline{x}=\tr(x)-x \in R$.
Therefore, $\overline{\mkern2mu\cdot\mkern2mu}$ restricts to an involution of $R$.

It is immediate from the definitions that $\cA(R^\bullet) = \cA((R^\bullet)^{\text{op}})$.
Since conjugation provides an isomorphism $R^\bullet \to (R^\bullet)^{\text{op}}$, it follows that $x \in R^\bullet$ is an atom if and only if $\overline x$ is an atom.

We denote by $\val=\val_K\colon K \to \bZ \cup \{\infty\}$ the valuation on $K$.
For quaternion orders we are able to relate the $k$-th elasticities to the maximal valuation of the norm of an atom.
The following lemma runs parallel to \cite[Proposition 6.1]{Geroldinger16}.

\begin{lemma} \label{l-rho}
  Let $M=\sup\{\, \val(\nr(u)) : u \in \cA(R^\bullet) \,\} \in \bN \cup \{\infty\}$, let $m = \min\{\, \val(\nr(u)) : u \in \cA(R^\bullet) \,\}$, and let $\sD = 2M/m \in \bQ_{\ge 2} \cup \{\infty\}$.
  \begin{enumerate}
  \item \label{l-rho:denom} If $\pi \in \cA(R^\bullet)$, then $m=2$ and otherwise $m=1$.
    Moreover, $\sL(\pi) = \{ 2/m \}$ and $\sD \in \bN \cup \{\infty\}$.
  \item \label{l-rho:rhok} For $k \in \bN$, $\rho_{2k}(R^\bullet) = k\sD$,
    \[
      1 + k\sD \le \rho_{2k+1}(R^\bullet) \le k\sD + \left\lfloor \frac{\sD}{2} \right\rfloor,
    \]
    and
    \[
      \rho(R^\bullet) = \frac{\sD}{2}.
    \]
  \end{enumerate}
  In particular, the following statements are equivalent:
  \begin{enumerate}[label=\textup{(}\alph*\textup{)}]
  \item\label{l-rho:2} $\rho_2(R^\bullet)=\infty$.
  \item\label{l-rho:allk} $\rho_k(R^\bullet)=\infty$ for all $k \ge 2$.
  \item\label{l-rho:onek} $\rho_k(R^\bullet)=\infty$ for some $k \ge 2$.
  \item\label{l-rho:elasticity} $\rho(R^\bullet)=\infty$.
  \item\label{l-rho:atom} For every $m \in \bN$, there exists an atom $u \in \cA(R^\bullet)$ such that $\val(\nr(u)) \ge m$ \textup{(}that is, $D=\infty$\textup{)}.
  \end{enumerate}
\end{lemma}

\begin{proof}
  \ref*{l-rho:denom}
  Since $R^\bullet$ is atomic, there exist $u_1$,~$\ldots\,$,~$u_l \in \cA(R^\bullet)$ such that $\pi = u_1 \cdots u_l$.
  Since $\nr(\pi)=\pi^2$, we must have $l \le 2$ and $\val(\nr(u_i)) \le 2$ for all $i \in [1,l]$.
  Thus, $m \le 2$.
  If $\pi \not\in\cA(R^\bullet)$, then $l = 2$ and hence $m=1$.
  Conversely, if $m=1$ and $u \in \cA(R^\bullet)$ with $\val(\nr(u)) =1$, then $u \overline u = \nr(u) = \varepsilon \pi$ for some $\varepsilon \in D^\times$ and so $\pi \not \in \cA(R^\bullet)$.
  Note also that $\sL(\pi) = \{2/m\}$ and $\sD \in \bN \cup \{\infty\}$.

  \ref*{l-rho:rhok}
  We first determine an upper bound for $\rho_k(R^\bullet)$.
  Suppose $k$, $l \in \bN$ with $k \le l$ and that $a=u_1\cdots u_k = v_1\cdots v_l$ with $a \in R^\bullet$, $u_1$, $\ldots\,$,~$u_k$, $v_1$, $\ldots\,$,~$v_l \in \cA(R^\bullet)$.
  Then $lm \le \val(\nr(a)) \le kM$ and hence $l \le kM/m$.
  We can conclude that $\rho_k(R^\bullet) \le \lfloor k M/m \rfloor$ and so $\rho_{2k}(R^\bullet) \le k \sD$ and $\rho_{2k+1}(R^\bullet) \le k\sD + \lfloor \sD/2\rfloor$.

  Let $u \in \cA(R^\bullet)$. Then $\nr(u) = \varepsilon \pi^l$ for some $l \in \bN$ with $l \le M$ and $\varepsilon \in D^\times$.
  Since $(u \overline{u})^k = \nr(u)^k = \varepsilon^k \pi^{lk}$, we find $\max \sL((u \overline{u})^k) \ge kl \max\sL(\pi)$.
  We can conclude that $\rho_{2k}(R^\bullet) \ge kM \max \sL(\pi) = k \sD$.
  Similarly, $(u\overline{u})^ku = \varepsilon^k \pi^{lk} u$ shows that $\rho_{2k+1}(R^\bullet) \ge lk \max\sL(\pi) + 1$ and hence $\rho_{2k+1}(R^\bullet) \ge k\sD + 1$.

  Finally, $\rho(R^\bullet) = \lim_{k\to \infty} \frac{\rho_{2k}(R^\bullet)}{2k} = \sD /2$.
\end{proof}

\section{Orders of \texorpdfstring{$A$}{A} when \texorpdfstring{$\widehat A$}{\^A} is a division ring}
\label{s-division}

We begin by considering the case where $D$ is a DVR and $A$ is a finite dimensional division ring over the quotient field $K$ of $D$, having the additional property that the completion $\widehat A$ is also a division ring.
This case is particularly easy since the valuation on $K$ extends to $A$.

\begin{theorem} \label{t-div}
  Let $D$ be a DVR, let $K$ be its quotient field, and let $A$ be a finite-dimensional central division ring over $K$.
  Suppose, moreover, that the completion $\widehat A \cong A \otimes_K \widehat K$ is also a division ring.
  If $R$ is an order in $A$, then
  \begin{enumerate}
    \item\label{t-div:rho} $\rho(R^\bullet) < \infty$ and hence $\rho_k(R^\bullet) < \infty$ for all $k \in \bN$,
    \item\label{t-div:cat} $\sc_\sd(R^\bullet) < \infty$ for any distance $\sd$ on $R^\bullet$,
    \item\label{t-div:delta} $\card{\Delta(R^\bullet)} < \infty$,
    \item\label{t-div:rhokgaps} there exists $M \in \bN$ such that $\rho_k(R^\bullet) - \rho_{k-1}(R^\bullet) \le M$ for all $k \in \bN_{\ge 2}$,
    \item\label{t-div:structure}
      the Structure Theorem for Unions of Sets of Lengths holds for $R^\bullet$.
    \end{enumerate}
\end{theorem}

\begin{proof}
  We briefly summarize the essential properties of the extension of the valuation $\mathsf v_K$ to $A$.
  These results follow from \cite[\S12 and \S13]{Reiner} when $A=\widehat A$ and descend to the non-complete case by \cite[\S18]{Reiner}.

  Since $\widehat A$ is a division ring and $\widehat D$ is complete, there exists a valuation $\val_{\widehat A} \colon \widehat A \to \bZ \cup \{\infty\}$ such that $\val_{\widehat A}(x) = \tfrac{e}{m} \val_{\widehat K}(\nr(x))$ where $m^2=\dim_{\widehat K} {\widehat A}=\dim_K A$ and $e=e(\widehat A/\widehat K)$ is the ramification index.
  This restricts to a valuation $\val_A = \val_{\widehat A}|_A$ on $A$.
  The ring
  \[
  S = \{\, x \in A : \val_A(x) \ge 0 \,\}
  \]
  is the unique maximal $D$-order in $A$.
  Let $\gamma \in S$ with $\val_A(\gamma)=1$.
  Every element of $A$ may be represented in the form $\gamma^n \varepsilon$ with $n \in \bZ$ and $\varepsilon \in S^\times$.
  Since $S$ and $R$ are equivalent orders, there exist $a$, $b \in S$ such that $aSb \subset R$.
  Writing $a=\gamma^{\val_A(a)}\varepsilon$ and $b=\delta \gamma^{\val_A(b)}$ with $\varepsilon$,~$\delta \in S^\times$ and using $\gamma S = S \gamma$, we see $\gamma^{\val(a)+\val(b)}S \subset R$.
  It follows that there exists a minimal $n \in \bN_0$ such that $\gamma^n S = \{\, x \in A : \val_A(x) \ge n \,\} \subset R$.
  In particular:
  \begin{itemize}
    \item $\gamma^n \in R$, and
    \item If $a$,~$b \in R^\bullet$ such that $\val_A(b) + n \le \val_A(a)$, then there exists $c \in R^\bullet$ with $a=bc$.
  \end{itemize}
  In particular, if $u \in \cA(R^\bullet)$, then $\val_A(u) \le 2n-1$.
  By \subref{l-basic:unit}, an element $a \in R^\bullet$ is a unit if and only if $\val_A(a)=0$.
  We are now able to prove the claims of the theorem.

  \ref*{t-div:rho}
  Let $k, l \in \bN$ with $k\leq l$ and let $u_1$,~$\ldots\,$,~$u_k$, $v_1$, $\ldots\,$,~$v_l \in \cA(R^\bullet)$ be such that $a= u_1\cdots u_k = v_1\cdots v_l$.
  Then $l \le \val_A(a) \le k(2n-1)$, and hence $\tfrac{l}{k} \le 2n-1$.
  It follows that $\rho(R^\bullet) \le 2n-1$.

  \ref*{t-div:cat}
  Let $\sd$ be a distance on $R^\bullet$ and let $a \in R^\bullet$.
  If $a \in R^\times$, then $\sc_\sd(a)=0$, and hence we may suppose $a \not \in R^\times$.
  We claim $\sc_\sd(a) \le 3n-1$.
  Suppose that $z=\rf{u_1,\cdots,u_k}$ and $z'=\rf{v_1,\cdots,v_l}$ with $u_1$, $\ldots\,$,~$u_k$, $v_1$, $\ldots\,$,~$v_l \in \cA(R^\bullet)$ are two factorizations of $a$.
  We need to show that there exists a $(3n-1)$-chain between $z$ and $z'$.
  Since $\val_A(v_i) \ge 1$ for all $i \in [1,l]$ and $\val(u_1) \le 2n-1$, there exists $m \in [1,l]$ with $m \le 3n-1$ such that $v_1\cdots v_m = u_1 c$ with $c \in R^\bullet$.
  Taking $m$ minimal, we must have $\val_A(v_1\cdots v_{m-1}) < \val_A(u_1) + n$.
  Since $\val_A(v_m) \le 2n-1$, it follows that $\val_A(c) < 3n-1$.
  Hence $c=w_1\cdots w_r$ with $r \in [1,3n-2]$ and $w_1$, $\ldots\,$,~$w_r \in \cA(R^\bullet)$.
  Thus
  \[
  \sd(\rf{u_1,w_1,\cdots,w_r,v_{m+1},\cdots,v_l},\, \rf{v_1,\cdots,v_l}) \le 3n-1.
  \]
  Since $w_1\cdots w_r v_{m+1} \cdots v_l = u_2\cdots u_k$, we can iterate this process to find a $(3n-1)$-chain between $z$ and $z'$.

  \ref*{t-div:delta}
  $\max \Delta(R^\bullet) \le \sc_\sd(R^\bullet) < \infty$ by \cite[Lemma 4.2]{Baeth-Smertnig}.

  \ref*{t-div:rhokgaps}
  Let $k$,~$l \in \bN_{\ge 2}$ and $u_1$,~$\ldots\,$,~$u_k$, $v_1$, $\ldots\,$,~$v_l \in \cA(R^\bullet)$ be such that $u_1\cdots u_k = v_1\cdots v_l$.
  As in \ref*{t-div:cat}, there exists $m \le 3n-1$ such that $u_1$ left divides $v_1\cdots v_m$, say $v_1\cdots v_m = u_1 c$ with $c \in R^\bullet$.
  If $c \in R^\times$, then $m=1$, and $l-1 \le \rho_{k-1}(R^\bullet)$.
  If $c \not \in R^\times$, then $\max \sL(c) \ge 1$, and $u_2\cdots u_k = c v_{m+1} \cdots v_l$ implies
  \[
    1 + (l-m) \le \max \sL(c) + \max \sL(v_{m+1}\cdots v_l) \le \rho_{k-1}(R^\bullet).
  \]
  Thus $l \le \rho_{k-1}(R^\bullet) + 3n-2$.
  We conclude $\rho_k(R^\bullet) \le \rho_{k-1}(R^\bullet) + 3n-2$.

  \ref*{t-div:structure}
  Holds by \cite[Theorem 2.6]{Geroldinger16} since $\Delta(R^\bullet)$ is finite and \ref*{t-div:rhokgaps} holds.
\end{proof}

\begin{remark}\label{r-finitelyprimarymonoid}
  The proof of Theorem \ref{t-div} runs completely parallel to the one for finitely primary monoids of \ as given in \cite[Theorem 3.1.5]{GHK06} and \cite[Proposition 3.6]{Geroldinger-Kainrath09}. The fundamental property is that every atom $u \in \cA(R^\bullet)$ left divides any product of at least $3n-1$ elements. This can be viewed as a (very strong) variant of the $\omega$-invariant being bounded by $3n-1$. We leave as an open question whether or not there is a transfer homomorphism from $R^\bullet$ to some finitely primary monoid of rank $1$.
\end{remark}

\section{Eichler orders in \texorpdfstring{$M_2(K)$}{M\texttwoinferior(K)}}
\label{s-eichler}

Let $D$ be a DVR (not necessarily complete) with prime element $\pi$ and valuation $\val$, and let $K$ be the quotient field of $D$.
We now consider the case where $R$ is an Eichler order in $A=M_2(K)$.
Thus, without restriction,
\begin{equation}\label{d-Eichler}
R=\begin{bmatrix} D & \pi^nD \\ D & D \end{bmatrix}=\left\{\,\begin{bmatrix} a & b\pi^n \\ c & d \end{bmatrix} : a,b,c,d \in D\,\right\},
\end{equation}
for some $n \in \bN_0$.
The number $n$ is referred to as the \emph{level} of the Eichler order $R$.
The reduced norm coincides with the determinant, and conjugation coincides with the adjugate.

In case $n \in \{0,1\}$, the ring $R$ is a hereditary order, and hence the results of \cite{Estes,Smertnig17} apply.
Before proceeding, we briefly summarize the known results in these cases.
If $n=0$, then $R=M_2(D)$ is a maximal order.
Since $R$ is a principal ideal ring, $R^\bullet$ is similarity factorial; that is, as a consequence of the Jordan-H\"{o}lder Theorem, factorizations are unique up to order and similarity of the atoms.
Moreover, since every $A\in R$ has a Smith Normal Form, the determinant is a transfer homomorphism.

If $n=1$, then $R$ is a non-maximal hereditary order.
The determinant is again a transfer homomorphism (see \cite{Estes}).
Moreover, $R^\bullet$ is composition series factorial, but not similarity factorial (see \cite[Proposition 4.12]{Smertnig17}).

In particular, if $n\in \{0,1\}$, then $R^\bullet$ is half-factorial and so $\rho(R^\bullet)=1$ and $\Delta(R^\bullet)=\emptyset$.
It is also known that $\sc_\sd(R^\bullet) \le 2$ for every distance $\sd$ on $R^\bullet$.
(This follows from \cite[Theorem 4.10]{Smertnig17} together with the fact that $R^\bullet$ has trivial class group.)

In this section, we study the remaining cases, in which $R$ is non-hereditary.
\begin{center}
  \emph{For the remainder of this section, we assume $R$ is as in \cref{d-Eichler} and $n \ge 2$.}
\end{center}

For $A \in R$, let $\val_{i,j}(A)$ denote the valuation of the $(i,j)$-th entry of $A$.

We shall repeatedly make use of the fact that for $2\times 2$ matrices, taking the adjugate, which we denote by $\adj$, is the standard involution on $M_2(R)$.
This often allows us to treat cases involving $\val_{1,1}(A)$ and $\val_{2,2}(A)$ by symmetry.

Note that, by \cref{l-basic}, we have $R^\bullet = \{\, A \in R : \det(A) \ne 0 \,\}$ and $R^\times = \{\, A \in R : \det(A) \in D^\times \,\}$.
The units can be described even more explicitly.
\begin{lemma} \label{l-units-1}
  Let
  \[
  A = \begin{bmatrix} a & b\pi^n \\ c & d \end{bmatrix} \in R.
  \]
  Then $A \in R^\times$ if and only if $a$,~$d \in D^\times$.
\end{lemma}

\begin{proof}
  We have $\det(A) = ad - bc\pi^n$ and $A$ is a unit if and only if $\det(A) \in D^\times$, that is, $\val(\det(A))=0$.
  Suppose $A \in R^\times$.
  Since $\val(ad - bc\pi^n) \ge \min \{ \val(a) + \val(d), \val(b) + \val(c) + n \}$ we must have $\val(a)=\val(d)=0$. Thus $a$,~$d \in D^\times$.

  Conversely, if $\val(a)=\val(d) = 0$, then $0 = \val(ad) < \val(bc\pi^n)$ and hence $\val(\det(A)) = \min\{\val(ad), \val(bc\pi^n) \} = 0$.
  Thus $\det(A) \in D^\times$ and so $A \in R^\times$.
\end{proof}

\Subref{l-basic:atom} also implies that if $A \in R^\bullet$ with $\det(A) \in \cA(D^\bullet)$, then $A \in \cA(R^\bullet)$.
We now prove a partial converse.
\begin{lemma} \label{l-det-atom}
  Let
  \[
  A=\begin{bmatrix} a & b\pi^n \\ c & d\end{bmatrix} \in R^\bullet
  \]
  with $\val(b)+\val(c) > 0$.
  Then $A$ is an atom in $R^\bullet$ if and only if $\det(A)$ is an atom in $D^\bullet$.
\end{lemma}
\begin{proof}
  We have already noted that if $\det(A)$ is an atom of $D^\bullet$, then $A$ is an atom of $R^\bullet$.
  Suppose to the contrary that $\det(A)=ad-bc\pi^n$ is reducible, that is $\val(ad) \ge 2$.
  Suppose that $\val(a) \ge 1$.
  If $\val(b) \ge 1$, then
  \[
  A = \begin{bmatrix} \pi & 0 \\ 0 & 1 \end{bmatrix} \begin{bmatrix} \pi^{-1}a & b\pi^{-1}\pi^n \\ c & d \end{bmatrix}
  \]
  is a factorization of $A$ into two non-units.
  If $\val(c) \ge 1$, then
  \[
  A = \begin{bmatrix} \pi^{-1}a & b\pi^n \\ \pi^{-1}c & d \end{bmatrix} \begin{bmatrix} \pi & 0 \\ 0 & 1 \end{bmatrix}.
  \]
  If $\val(a) = 0$, then $\val(d) \ge 2$, and by what we have already shown, $\adj(A)$ is reducible.
  Thus $A=\adj(\adj(A))$ is also reducible.
\end{proof}
However, we now show that --- completely contrary to the cases $n \in \{0,1\}$ --- there exist atoms whose determinant is not an atom, and thus the converse of \subref{l-basic:atom} is false in general.

\begin{lemma} \label{l-large-atom}
  If
  \[
  A=\begin{bmatrix} a & b\pi^n \\ c & d \end{bmatrix} \in R^\bullet
  \]
  with $\val(b)=\val(c)=0$ and $\val(a) > 0$, $\val(d) > 0$, then $A$ is an atom in $R^\bullet$.
\end{lemma}

\begin{proof}
Suppose that $A$ factors as
\begin{equation}\label{eq-aprod}
  A=\begin{bmatrix} a & b\pi^n \\ c & d\end{bmatrix}
  = \underbrace{\begin{bmatrix} e & f\pi^n \\ g & h \end{bmatrix}}_B \underbrace{\begin{bmatrix} p & q\pi^n \\ r & s \end{bmatrix}}_C
  = \begin{bmatrix} ep + fr \pi^n & (eq + fs) \pi^n \\ gp + hr & gq\pi^n + hs \end{bmatrix},
\end{equation}
a product of two non-units $B$ and $C$ of $R^\bullet$.
Considering the upper-right and lower-left corners of $A$, we obtain
\[
0 = \val(b) = \min\{ \val(e) + \val(q), \val(f) + \val(s) \}
\]
and
\[
0 = \val(c) = \min\{ \val(g) + \val(p), \val(h) + \val(r) \}.
\]

First suppose that $\val(e) + \val(q) = 0$.
Then $\val(h) > 0$ since $B$ is not a unit.
Thus $\val(g) + \val(p) = 0$.
But then
\[
\val(a)=\val(ep+fr\pi^n)=\min\{\val(e) + \val(p), \val(f)+\val(r)+n\} = 0,
\]
a contradiction to the assumption that $\val(a) > 0$.

In the case that $\val(f) + \val(s) = 0$, we similarly find $\val(p) > 0$ and hence $\val(h)+\val(r)=0$, a contradiction to $\val(d) > 0$.
\end{proof}

We will soon classify all of the atoms of $R^\bullet$.
Before doing so, we give a lemma that provides atoms of arbitrarily large valuation.
This, in turn, guarantees infinite elasticity as per \subref{l-rho:rhok}.

\begin{lemma} \label{l-eichler-longatom}
  For all $a \in D^\bullet \setminus D^\times$, there exists an atom $U \in \cA(R^\bullet)$ with $\det(U)=a$.
  In particular, we have $\rho_k(R^\bullet)=\infty$ for all $k \ge 2$.
\end{lemma}

\begin{proof}
  If $\val(a)=1$, then
  \[
  U=\begin{bmatrix} a & 0 \\ 0 & 1 \end{bmatrix}
  \]
  is an atom with $\det(U)=a$.
  If $\val(a) > 1$, then \cref{l-large-atom} implies that
  \[
  U=\begin{bmatrix} a\pi^{-1} + \pi^{n-1} & \pi^n \\ 1 & \pi \end{bmatrix}
  \]
  is an atom. Moreover, $\det(U) = a$. The final statement now follows from \subref{l-rho:rhok}.
\end{proof}

Unlike for the rings studied in Section \ref{s-noneichler}, we are able to give a full classification of the atoms of Eichler orders. Before classifying the atoms of $R$, we take a closer look at the unit group $R^\times$ and investigate how the valuations of entries of matrices in $R$ behave with respect to products.
\begin{proposition} \label{p-unitgroup}
  Let
  \[
  L = \left\{\, \begin{bmatrix} 1 & 0 \\ x & 1 \end{bmatrix} : x \in D \,\right\},\;
  U = \left\{\, \begin{bmatrix} 1 & x\pi^n \\ 0 & 1 \end{bmatrix} : x \in D \,\right\},\;
  \Delta = \left\{\, \begin{bmatrix} \varepsilon & 0 \\ 0 & \eta \end{bmatrix} : \varepsilon, \eta \in D^\times \,\right\}.
  \]
  Then $R^\times = L\Delta U$.
\end{proposition}

\begin{proof}
  Let
  \[
  A = \begin{bmatrix} a & b\pi^n \\ c & d \end{bmatrix} \in R^\times.
  \]
  Then $a$, $d \in D^\times$ by \cref{l-units-1} and
  \[
A = \begin{bmatrix} 1 & 0 \\ ca^{-1} & 1\end{bmatrix} \begin{bmatrix} a & 0 \\ 0 & d - ca^{-1}b\pi^n \end{bmatrix} \begin{bmatrix} 1 & a^{-1}b\pi^n \\ 0 & 1 \end{bmatrix}
  \]
  gives the desired decomposition.
\end{proof}

In general, the valuations of the individual entries of a matrix are not additive on products.
Moreover, they are not preserved under associativity.
However, as we shall observe, the valuations in the upper-left and the lower-right corners are preserved under products and associativity if they are small enough.

\begin{lemma} \label{l-add}
  Let $A=A_1 \cdots A_m$ with $A$, $A_1$,~$\ldots\,$,~$A_m \in R^\bullet$.
  \begin{enumerate}
  \item \label{l-add:11} If $\val_{1,1}(A) < n$ then $\val_{1,1}(A) = \val_{1,1}(A_1) + \cdots + \val_{1,1}(A_m)$.
  \item \label{l-add:22} If $\val_{2,2}(A) < n$ then $\val_{2,2}(A) = \val_{2,2}(A_1) + \cdots + \val_{2,2}(A_m)$.
  \end{enumerate}
\end{lemma}

\begin{proof}
  It suffices to consider the case $m=2$; the general case follows by induction.
  Let $A=BC$ be as in \cref{eq-aprod}.
  Then $\val_{1,1}(A) = \val(a) = \val(ep + fr\pi^n)$.
  By assumption, $\val(a) < n$ and therefore $\val(a) = \val(ep) = \val_{1,1}(B) + \val_{1,1}(C)$.
  The second claim follows by symmetry.
\end{proof}

\begin{lemma} \label{l-assoc-v}
  Let $A$,~$A' \in R^\bullet$ be associated.
  \begin{enumerate}
  \item\label{l-assoc-v:11} If $\val_{1,1}(A) < n$, then $\val_{1,1}(A') = \val_{1,1}(A)$.
  \item\label{l-assoc-v:22} If $\val_{2,2}(A) < n$, then $\val_{2,2}(A') = \val_{2,2}(A)$.
  \end{enumerate}
\end{lemma}

\begin{proof}
  We prove \ref*{l-assoc-v:11} and  \ref*{l-assoc-v:22} follows by symmetry.
  By assumption $A=BA'C$ with $B$,~$C \in R^\times$.
  Then \cref{l-add} yields $\val_{1,1}(A) \ge \val_{1,1}(A')$.
  This implies $\val_{1,1}(A') < n$, and applying \cref{l-add} to $A'=B^{-1}AC^{-1}$ yields $\val_{1,1}(A') \ge \val_{1,1}(A)$.
\end{proof}

We are now ready to give a characterization of all atoms of $R^\bullet$.

\begin{theorem} \label{t-atom}
  Let $D$ be a DVR with quotient field $K$, $n \in \bN_{\ge 2}$, and
  \[
    R = \begin{bmatrix} D & \pi^n D \\ D & D \end{bmatrix} \subset M_2(D)
  \]
  an Eichler order of level $n$ in $M_2(K)$.

  An element
  \[
  A = \begin{bmatrix} a & b\pi^n \\ c & d \end{bmatrix} \in R^\bullet
  \]
  is an atom of $R^\bullet$ if and only if one of the following holds:
  \begin{enumerate}[label=\textup{(\Roman*)}]
  \item\label{t-atom:i} $\val(\det(A)) = 1$ \textup{(}equivalently, $\sv(ad)=1$\textup{)}.
  \item\label{t-atom:ii} $\val(b)=\val(c)=0$ and $\val(a) > 0$, $\val(d) > 0$.
  \end{enumerate}
\end{theorem}

\begin{proof}
  From \subref{l-basic:atom} and \cref{l-large-atom} we already know that $A$ is an atom if it has either of the two stated forms.
  Suppose now that $A$ is an atom.
  We assume $\val(\det(A)) \ge 2$ and show that \ref*{t-atom:ii} holds.
  By \cref{l-det-atom} we must have $\val(b)=\val(c)=0$.
  It remains to show that $\val(a) > 0$ and $\val(d) > 0$.

  Assume to the contrary that $\val(a) = 0$.
  Then $\val(d) \ge 2$ and
  \[
  \begin{bmatrix} 1 & 0 \\ -ca^{-1} & 1 \end{bmatrix} A = \begin{bmatrix} a & b\pi^n \\ 0 & d - ca^{-1}b\pi^n \end{bmatrix}.
  \]
  Applying \cref{l-det-atom}, we see that the matrix on the right hand side of the equation is not an atom since $\val(d - ca^{-1}b\pi^n) \ge \min\{\val(d),n\} \ge 2$.
  Thus $A$, an associate of this matrix, is also not an atom.

The case $\val(d) = 0$ follows by symmetry. Thus $\val(a), \val(d)>0$.
\end{proof}

Note that the only atoms $A$ with $\val(\det(A)) > n$ are those of type \ref{t-atom:ii} with $\val(a) + \val(d) = n$.
For this type of atom, the determinant can have arbitrarily large valuation as we have already seen in \cref{l-eichler-longatom}.

The Jacobson radical of $R$ is
\[
J(R)=
\begin{bmatrix}
  \pi D & \pi^n D \\
  D     & \pi D
\end{bmatrix}.
\]
Thus the atoms of type \ref{t-atom:ii} are precisely the ones contained in $J(R)$.

\begin{corollary} \label{c-atomr}
For $m \in \bN_0$, let $\cR(m) \subset D$ be a system of representatives for $D/\pi^m D$.
Every atom of $R^\bullet$ is right associated to precisely one of the following atoms:
\begin{enumerate}
\item\label{c-atomr:1}
  $\displaystyle
  \begin{bmatrix} \pi & \lambda \pi^n \\ 0 & 1 \end{bmatrix} \qquad \text{with $\lambda \in \cR(1)$}.
  $
\item\label{c-atomr:2}
  $\displaystyle
  \begin{bmatrix} 1 & 0 \\ \lambda & \pi \end{bmatrix} \qquad \text{with $\lambda \in \cR(1)$}.
  $
\item\label{c-atomr:3}
  $\displaystyle
  \begin{bmatrix} \varepsilon \pi^m & \pi^n \\ 1 & \delta \pi^{m'}\end{bmatrix}
  $
  \quad
  with  $1\leq m, m'<n$ such that $m+m' < n$, $\varepsilon \in D^\times \cap \cR(m')$, and $\delta \in D^\times \cap \cR(m)$.
\item\label{c-atomr:4}
  $\displaystyle
  \begin{bmatrix} \varepsilon \pi^m & \pi^n \\ 1 & \delta \pi^{m'}\end{bmatrix}
  $
  \quad
  with  $1\leq m, m'<n$ such that $m+m' > n$, $\varepsilon \in D^\times \cap \cR(n-m)$, and $\delta \in D^\times \cap \cR(n-m')$.
\item\label{c-atomr:5}
  $\displaystyle
  \begin{bmatrix} \varepsilon \pi^m & \pi^n \\ 1 & 0\end{bmatrix}
  $
  \quad
  with $1 \le m < n$ and $\varepsilon \in D^\times \cap \cR(n-m)$.
\item\label{c-atomr:6}
  $\displaystyle
  \begin{bmatrix} 0 & \pi^n \\ 1 & \delta \pi^{m'}\end{bmatrix}
  $
  \quad
  with $1 \le m' < n$ and $\delta \in D^\times \cap \cR(n-m')$.
\item\label{c-atomr:7}
  $\displaystyle
  \begin{bmatrix} 0 & \pi^n \\ 1 & 0 \end{bmatrix}.
  $
\item\label{c-atomr:8}
  $\displaystyle
  \begin{bmatrix} \varepsilon \pi^m & \pi^n \\ 1 & (\varepsilon^{-1} + \pi^k \delta)\pi^{m'}\end{bmatrix}
  $
  \quad
  with $1 \leq m, m' < n$ such that $n = m+m'$,  $k \in \bN_0$, $\varepsilon \in D^\times \cap \cR(m'+k)$, and $\delta \in D^\times \cap \cR(m)$.
\end{enumerate}
\end{corollary}

\begin{proof}
From \cref{l-assoc-v}, we can immediately note that if $A$ belongs to class ($i$) and $B$ belongs to class ($j$) with $i\not=j$, then $A$ is not associated to $B$.

  Let
  \[
    U = \begin{bmatrix} a & b \pi^n \\ c & d \end{bmatrix}
  \]
  be an atom of $R^\bullet$.

  Suppose that $U$ is of type \labelcref{t-atom:i}.
  If $\val(a)=1$ and $\val(d) = 0$, then
  \[
  U \begin{bmatrix} 1 & 0 \\ 0 & d^{-1} \end{bmatrix} \begin{bmatrix} 1 & 0 \\ -c & 1 \end{bmatrix} = \begin{bmatrix} a - bcd^{-1}\pi^n & bd^{-1} \pi^n \\ 0 & 1 \end{bmatrix}.
  \]
  Now $\val(a-bcd^{-1}\pi^n) = \val(a) = 1$ and thus, by dividing the first column by a suitable unit of $D^\times$, we see that $U$ is right associated to
  \[
    \begin{bmatrix} \pi & bd^{-1} \pi^n \\ 0 & 1 \end{bmatrix}.
  \]
  If $\lambda \in \cR(1)$ with $\lambda \equiv bd^{-1} \mod \pi D$ and $x \in D$ with $\lambda = bd^{-1} + \pi x$, then
  \[
    \begin{bmatrix} \pi & bd^{-1}\pi^n \\ 0 & 1 \end{bmatrix} \begin{bmatrix} 1 & x\pi^n \\ 0  & 1 \end{bmatrix} = \begin{bmatrix} \pi & \lambda \pi^n \\ 0 & 1 \end{bmatrix},
  \]
  so that $U$ is right associated to an atom in class \labelcref*{c-atomr:1}.

  We now show that no two atoms in class \labelcref*{c-atomr:1} are right associated to each other.
  Let $\lambda \in \cR(1)$, $e$,~$h \in D^\times$ and $f$,~$g \in D$.
  Then
  \[
    \begin{bmatrix} \pi & \lambda \pi^n \\ 0 & 1 \end{bmatrix} \begin{bmatrix} e & f\pi^n \\ g & h \end{bmatrix} = \begin{bmatrix} e\pi + g \lambda \pi^n & (f\pi + h \lambda) \pi^n \\ g & h \end{bmatrix}.
  \]
  For this to have the same form as a matrix in class \labelcref*{c-atomr:1}, we must have $g=0$, $h=1$, and subsequently $e=1$.
  Finally, since $f \pi + \lambda \equiv \lambda \mod \pi D$, the condition $f \pi + \lambda \in \cR(1)$ forces $f=0$.

  If $\val(a)=0$ and $\val(d) = 1$, then one shows analogously that $U$ is right associated to precisely one atom in class \labelcref*{c-atomr:2}.

  Suppose now that $U$ is an atom of type \labelcref{t-atom:ii}.
  That is, $\val(b)=\val(c)=0$ and $\val(a) > 0$, $\val(d) > 0$.
  Dividing the columns by suitable units we may, without restriction, assume
  \[
    U = \begin{bmatrix} a & \pi^n \\ 1 & d \end{bmatrix}.
  \]
  Now let $e$, $h \in D^\times$ and $f$, $g \in D$.
  Then
  \[
    U \begin{bmatrix} e & f\pi^n \\ g & h \end{bmatrix} =
    \begin{bmatrix}  ae + g\pi^n & (af+h)\pi^n \\ e+dg & f\pi^n + dh \end{bmatrix}.
  \]
  This right associate again has the same form as $U$ if and only if $1=e+dg=af+h$.
  Then $ae + g\pi^n = a + g (\pi^n - ad)$ and $f\pi^n + dh = d + f (\pi^n - ad)$.
  Note that $\pi^n - ad = -\det(U)$.
  Thus,
  \[
    \begin{bmatrix} a' & \pi^n \\ 1 & d' \end{bmatrix} \in R^\bullet
  \]
  with $a'$,~$d' \in D$ is a right associate of $U$ if and only if $a' \equiv a \mod \pi^{\val(\det(U))} D$ and $d' \equiv d \mod \pi^{\val(\det(U))}D$.

  If $\val(a) + \val(d) < n$, then $\val(\det(U)) = \val(a) + \val(d)$, and $U$ is right associated to precisely one of the atoms in class \labelcref*{c-atomr:3}.

  If $\val(a) + \val(d) > n$, then $\val(\det(U)) = n$, and $U$ is right associated to precisely one of the atoms in one of the classes \labelcref*{c-atomr:4,c-atomr:5,c-atomr:6,c-atomr:7}, depending on whether or not $\val(a) < n$ or $\val(d) < n$.

  Finally, suppose that $n = \val(a) + \val(d)$.
  Let $m=\val(a)$, $m' = \val(d)$, and $k = \val(\det(U)) - n \in \bN_0$.
  Then $a = \varepsilon \pi^m$ and $d = \varepsilon' \pi^{m'}$ with $\varepsilon$,~$\varepsilon' \in D^\times$.
  Note that $k=\val(1-\varepsilon\varepsilon')=\val(\varepsilon^{-1} - \varepsilon')$.
  Hence $\varepsilon' = \varepsilon^{-1} + \pi^k \delta$ for some $\delta \in D^\times$.
  Since $a$ and $d$ are determined up to congruence modulo $\pi^{n+k}$, we can pick $\varepsilon \in \cR(n+k-m) = \cR(m'+k)$ and $\delta \in \cR(n-m') = \cR(m)$.
  Thus $U$ is right associated to an atom in class \labelcref*{c-atomr:8}.
  The congruence condition also guarantees that no two atoms listed above are right associated.
\end{proof}

\begin{corollary} \label{c-atom}
  Any atom of $R^\bullet$ is \textup{(}two-sided\textup{)} associated to precisely one of
  \[
    \begin{bmatrix} \pi & 0 \\ 0 & 1 \end{bmatrix}, \begin{bmatrix} 1 & 0 \\ 0 & \pi \end{bmatrix},
  \]
  or one of the atoms in one of the classes \labelcref*{c-atomr:3,c-atomr:4,c-atomr:5,c-atomr:6,c-atomr:7,c-atomr:8} of \cref{c-atomr}.
\end{corollary}

\begin{proof}
  As we have already noted in the proof of \cref{c-atomr}, atoms listed in the different classes of \cref{c-atomr} are not associated.
  It is easy to see that all the atoms listed in class \labelcref*{c-atomr:1} are left associated.
  The same is true for the atoms listed in class \labelcref*{c-atomr:2}.

  For the remaining classes, it can be verified as in the proof of \ref{c-atomr} that none of the atoms listed are left associated.
  Alternatively, one may argue by symmetry using the involution:
  If two such atoms $U$ and $V$ are left associated, then $-\adj(U)$ and $-\adj(V)$ are right associated, and replacing the system of representatives $\cR(m)$ by $-\cR(m)$, again of a form as listed.
  Thus we conclude $U=V$.
\end{proof}

We now work towards a result on sets of lengths.
For this we need two more preparatory lemmas.
The first technical lemma says that $A' \in R^\bullet$ is always associated to a matrix in which either all components $\val_{2,1}(A')$, $\val_{1,1}(A')$,~$\val_{1,2}(A')$ are of roughly the same magnitude, or all components $\val_{2,1}(A')$, $\val_{2,2}(A')$,~$\val_{1,2}(A')$ are of roughly the same magnitude.

\begin{lemma} \label{l-special}
  Every $A' \in R^\bullet$ is associated to an element
  \[
    A = \begin{bmatrix} a & b\pi^n \\ c & d \end{bmatrix} \in R^\bullet
  \]
  satisfying
  \[
    \val(c) \le \min\{ \val(a), \val(d) \} \le \val(b) + n \le \min\{ \val(a), \val(d) \} + n \le \val(c) + 2n.
  \]
\end{lemma}

\begin{proof}
  By \cref{p-unitgroup}, beginning with a matrix $B \in R^\bullet$, we can add a multiple of the second column to the first column, we can add a $\pi^n$-multiple of the first column to the second column, we can add a multiple of the first row to the second row, and we can add a $\pi^n$-multiple of the second row to the first row, resulting in a matrix $B' \in R^\bullet$ that is associated to $B$.

  Since $\val_{1,2}(A) = \val(b) + n$, we need to show that $A'$ can be transformed into $A$ such that
  \[
    \begin{split}
      \val_{2,1}(A) &\le \min\{ \val_{1,1}(A), \val_{2,2}(A) \} \le \val_{1,2}(A)\\ &\le \min\{ \val_{1,1}(A), \val_{2,2}(A) \} + n \le \val_{2,1}(A) + 2n.
    \end{split}
  \]

  We proceed to transform $A'$ as follows.
  First, if $\val_{1,1}(A') > \val_{1,2}(A')$, we add the second column to the first one.
  Similarly, if $\val_{2,2}(A') > \val_{1,2}(A')$, we add the first row to the second one.
  This yields a matrix $A_1$ with $\val_{1,1}(A_1) \le \val_{1,2}(A_1)$ and $\val_{2,2}(A_1) \le \val_{1,2}(A_1)$.
  Now, if $\val_{2,1}(A_1) > \min\{ \val_{1,1}(A_1), \val_{2,2}(A_1) \}$, then, adding either the first row to the second row, or the second column to the first column (depending on which of $\val_{1,1}(A_1)$ and $\val_{2,2}(A_1)$ is minimal), yields a matrix $A_2$ with $\val_{2,1}(A_2) \le \min\{ \val_{1,1}(A_1), \val_{2,2}(A_1) \} \le \val_{1,2}(A_1)$.
  (If $\val_{2,1}(A_1) \le \min\{ \val_{1,1}(A_1), \val_{2,2}(A_1) \}$ we simply set $A_2=A_1$.)
  Note that $\min\{\val_{1,1}(A_2), \val_{2,2}(A_2) \} = \min\{ \val_{1,1}(A_1), \val_{2,2}(A_1) \}$, since the minimal value remains unchanged.
  Thus $A_2$ satisfies the first two inequalities, that is,
  \[
  \val_{2,1}(A_2) \le \min\{ \val_{1,1}(A_2), \val_{2,2}(A_2) \} \le \val_{1,2}(A_2).
  \]

  Now, if $\val_{1,1}(A_2) > \val_{2,1}(A_2) + n$, we add $\pi^n$ times the second row to the first.
  Similarly, if $\val_{2,2}(A_2) > \val_{2,1}(A_2) + n$, we add $\pi^n$ times the first column to the second.
  The resulting matrix $A_3$ satisfies $\val_{1,1}(A_3) \le \val_{2,1}(A_3)+n$ and $\val_{2,2}(A_3) \le \val_{2,1}(A_3) + n$ by construction.
  Thus $A_3$ satisfies the last inequality.
  Since the valuations in the upper-left and the lower-right corner cannot have increased, we have $\min\{ \val_{1,1}(A_3), \val_{2,2}(A_3) \} \le \min\{\val_{1,1}(A_2), \val_{2,2}(A_2) \}$.
  Thus
  \[
  \val_{2,1}(A_3) \le \min\{ \val_{1,1}(A_3), \val_{2,2}(A_3) \} \le \min\{\val_{1,1}(A_2), \val_{2,2}(A_2) \} \le \val_{1,2}(A_2).
  \]
  Now
  \[
    \begin{split}
      \val_{1,2}(A_3) &\ge \min\{ \val_{1,2}(A_2), \val_{1,1}(A_2) + n, \val_{2,2}(A_2) + n\}\\
      & \ge \min\{ \val_{1,2}(A_2), \val_{2,1}(A_2)+n \} \\
      & \ge \min\{ \val_{1,1}(A_3), \val_{2,2}(A_3) \}
    \end{split}
  \]
  and thus $A_3$ still satisfies the first two inequalities.

  If $\val_{1,2}(A_3) \le \min\{ \val_{1,1}(A_3), \val_{2,2}(A_3) \} + n$ we are done.
  Otherwise, suppose $\val_{1,1}(A_3) = \min\{ \val_{1,1}(A_3), \val_{2,2}(A_3) \}$ and $\val_{1,1}(A_3) + n  < \val_{1,2}(A_3)$. The other case is handled analogously.
  We then add $\pi^n$ times the first column to the second column to obtain a matrix $A_4$, which now satisfies all inequalities.
\end{proof}

Before considering sets of lengths, we need one final result about the associativity of atoms.

\begin{lemma} \label{l-intersection}
  Let $U$, $V \in R^\bullet$ be atoms that are not right associated.
  Then $UR \cap VR \subset J(R)$.
\end{lemma}

\begin{proof}
  If $U$ (respectively $V$) is an atom of type \labelcref{t-atom:ii}, then $U \in J(R)$ (respectively $V \in J(R)$) and we are done.
  We may now assume that $U$ and $V$ are atoms of type \labelcref{t-atom:i}.

  If two elements $U$, $U' \in R^\bullet$ are right associated, then $UR = U'R$, so it suffices to consider $U$ and $V$ of the form listed in \labelcref{c-atomr:1,c-atomr:2} of \cref{c-atomr}.

  For $\lambda \in D$, let
  \[
    W_1(\lambda) \coloneqq\begin{bmatrix} \pi & \lambda \pi^n \\ 0 & 1 \end{bmatrix}
    \quad\text{and}\quad
    W_2(\lambda) \coloneqq\begin{bmatrix} 1 & 0  \\ \lambda & \pi \end{bmatrix}.
  \]
  We have
  \[
    \begin{bmatrix} \pi & \lambda \pi^n \\ 0 & 1 \end{bmatrix}
    \begin{bmatrix} a & b\pi^n \\ c & d \end{bmatrix} =
    \begin{bmatrix}
      a \pi + c \lambda \pi^n & (b\pi + d \lambda)\pi^n \\
      c & d
    \end{bmatrix} \qquad(\text{for $a$, $b$, $c$, $d \in D$}),
  \]
  and conclude
  \[
    W_1(\lambda) R = \left\{\, \begin{bmatrix} a\pi  & b \pi^n \\ c & d \end{bmatrix} :  a, b, c, d \in D \text{ with } b \equiv d \lambda \mod \pi D \,\right\}.
  \]
  Similarly,
  \[
    W_2(\lambda) R = \left\{\, \begin{bmatrix} a  & b \pi^n \\ c & d\pi \end{bmatrix} :  a, b, c, d \in D \text{ with } c \equiv a \lambda \mod \pi D \,\right\}.
  \]
  Clearly $W_1(\lambda)R \cap W_2(\mu)R \subset J(R)$ for all $\lambda$,~$\mu \in D$.
  Due to the congruence condition, $W_1(\lambda) R \cap W_1(\mu)R \subset J(R)$ as well as $W_2(\lambda) R \cap W_2(\mu) R \subset J(R)$ for $\lambda$,~$\mu \in D$ with $\lambda \not\equiv \mu \mod \pi D$.
\end{proof}

\begin{theorem} \label{t-lengths}
  Let $D$ be a DVR with quotient field $K$, $n \in \bN_{\ge 2}$, and
  \[
    R = \begin{bmatrix} D & \pi^n D \\ D & D \end{bmatrix} \subset M_2(D)
  \]
  an Eichler order of level $n$ in $M_2(K)$. Let
  \[
  A = \begin{bmatrix} a & b\pi^n \\ c & d \end{bmatrix} \in R^\bullet.
  \]
  \begin{enumerate}
  \item\label{t-lengths:det} If $\val(a) = 0$ or $\val(d) = 0$ \textup{(}that is, $A \not\in J(R)$\textup{)}, then $\sL(A) = \{ \val(\det(A)) \}$ and $\card{\sZ^*(A)}=1$.

  \item\label{t-lengths:short} If $\val(a) > 0$ and $\val(d) > 0$ \textup{(}that is, $A \in J(R)$\textup{)}, then $\min \sL(A) \le n+5$.
  \end{enumerate}
\end{theorem}

\begin{proof}
  \ref*{t-lengths:det}
  Suppose $\val(a) = 0$; the case $\val(d)=0$ follows by symmetry.
  Let $A = U_1\cdots U_m$ with $m\ge 1$ and $U_1$, $\ldots\,$,~$U_m \in \cA(R^\bullet)$.
  By \subref{l-add:11} we have $\val_{1,1}(U_i)=0$ for all $i \in [1,m]$.
  Thus each $U_i$ is an atom of type \ref{t-atom:i}, which implies $\val(\det(U_i))=1$.
  Thus $m= \val(\det(U_1)) + \cdots + \val(\det(U_m)) = \val(\det(A))$.

  We show $\card{\sZ^*(A)}=1$ by induction on $m$, the length of the factorization $A=U_1\cdots U_m$.
  For $m=0$ this is trivially true since $A\in R^\times$.
  Suppose now $m > 0$ and that the claim has been established for $m-1$.
  Let
  \[
    A=U_1\cdots U_m = V_1\cdots V_m \quad\text{with}\quad U_1, \ldots\,,~U_m, V_1, \ldots\,,~V_m \in \cA(R^\bullet).
  \]
  Since $A \subset U_1R \cap V_1R$ and $A \not\in J(R)$, \cref{l-intersection} implies that $U_1$ and $V_1$ are right associated.
  Let $E_1 \in R^\times$ be such that $V_1 =  U_1 E_1$.
  Then $U_2 \cdots U_m = (E_1 V_2) V_3 \cdots V_m$.
  By \subref{l-add:11}, we have $\val_{1,1}(U_2\cdots U_m) = 0$.
  The induction hypothesis therefore implies $\rf{U_2,\cdots,U_m} = \rf{E_1V_2,V_3,\cdots,V_m}$.
  Moreover,
  \[
    \rf{U_1,U_2,\cdots,U_m} = \rf{V_1E_1^{-1},E_1V_2,V_3,\cdots,V_m} = \rf{V_1,\cdots,V_m}.
  \]

  \ref*{t-lengths:short}
  By \cref{l-special} we may, without restriction, assume
  \[
  \val(c) \le \min\{ \val(a), \val(d) \} \le \val(b) + n \le \min\{ \val(a), \val(d) \} + n \le \val(c) + 2n.
\]
  In particular, we have $\val(a) \ge \max\{\val(b),\val(c)\}$ and $\val(d) \ge \max\{\val(b),\val(c)\}$.
  Choose $m = \min\{\val(c),\val(b), \val(a)-1, \val(d)-1\}$ and consider
  \[
  A' = \begin{bmatrix} a' & b'\pi^n \\ c' & d' \end{bmatrix} = \pi^{-m} A \in R^\bullet.
  \]

  First note that $\min \sL(\pi^m) \le 2$:
  The case $m=0$ is trivial.
  For $m \ge 1$, \cref{l-eichler-longatom} implies that there exists an atom $U$ with $\det(U)=\pi^m$.
  Then $\adj(U)$ is also an atom and $\pi^m = U \adj(U)$.

  Now we show that $A'$ has a factorization of length at most $n+3$.
  Since $\val(a) \ge \max\{\val(b), \val(c)\}$ and $\val(d) \ge \max\{\val(b),\val(c)\}$, we have that $\val(c') \le 1$ or $\val(b') \le 1$.
  If $\val(b') \le 1$, then $\val(c') \le \val(b') + n = n+1$.
  If $\val(c') \le 1$, then $\val(b') \le \val(c') + n = n+1$.
  Thus, in either case $\val(b') + \val(c') + 1 \le n+3$ and it suffices to show that $A'$ has a factorization of length at most $\val(b') + \val(c') + 1$.

  Our choice of $m$ also ensures that $\val(a') \ge 1$ and $\val(d') \ge 1$.
  Since $\val(b') \le 1$ or $\val(c') \le 1$ and both of these two values are bounded by $\min\{\val(a'), \val(d')\}$, we have $\val(b') + \val(c') \le \min\{\val(a'), \val(d')\} + 1 \le \val(a') + \val(d')$.

  Suppose first that $\val(b') + \val(c') \ge \val(a') + \val(d') - 1$ and set $l = \min\{\val(a'),\val(b')\}$.
  Since $\val(a') \le \val(b') + \val(c')$, we have $\val(a') - l \le \val(c')$, and hence
  \[
  A' =
  \begin{bmatrix} \pi^{l} & 0 \\ 0 & 1 \end{bmatrix}
  \underbrace{\begin{bmatrix} a' \pi^{-\val(a')} & b' \pi^{-l}\pi^n \\ c' \pi^{-\val(a') + l} & d' \end{bmatrix}}_B
  \begin{bmatrix} \pi^{\val(a')-l} & 0 \\ 0 & 1 \end{bmatrix}.
  \]
  Now $\sL(B) = \{ \val(d') \}$ by \ref*{t-lengths:det}.
  Thus,  $A'$ has a factorization of length at most $\val(a') + \val(d') \le \val(b') + \val(c') + 1$.

  If $\val(b') + \val(c') \le \val(a') + \val(d') - 2$, we can factor out $\val(b') + \val(c')$ atoms of type \ref*{t-atom:i}, in a way that ensures that the upper-left and lower-right corners of the remaining matrix still have positive valuations, and also so that the upper-right and lower-left corners have valuation $0$.
  What remains is therefore an atom of type \ref*{t-atom:ii}.
  This again gives a factorization of $A'$ of length at most $\val(b') + \val(c') + 1$.
\end{proof}

\begin{lemma} \label{l-mindelta}
  $\min\Delta(R^\bullet) = 1.$
\end{lemma}

\begin{proof}
  With
  \[
  A = \begin{bmatrix} \pi & \pi^n \\ \pi & \pi^2 + \pi^n \end{bmatrix}
  \]
  we have
  \[
  A = \begin{bmatrix} 1 & 0 \\ 0 & \pi \end{bmatrix} \begin{bmatrix} \pi & \pi^n \\ 1 & \pi + \pi^{n-1} \end{bmatrix},
  \]
  where both factors are atoms.
  Also,
  \[
  A = \begin{bmatrix} 1 & \pi^n \\ 1 & \pi^2 + \pi^n \end{bmatrix} \begin{bmatrix} \pi & 0 \\ 0 & 1 \end{bmatrix},
  \]
  where the second factor is an atom, and the first factor has the unique factorization length $2$ by \subref{t-lengths:det}.
  Thus $\sL(A)=\{2,3\}$ and so $1 \in \Delta(A) \subset \Delta(R^\bullet)$.
\end{proof}

From the last two results, it is easy to deduce the factorization-theoretic properties for Eichler orders in $M_2(K)$ that we claim in the introduction.
We do so in \cref{s-wrapup} in a more general setting.

\section{Non-Eichler orders in \texorpdfstring{$M_2(K)$}{M\texttwoinferior(K)}}
\label{s-noneichler}

In this section we consider the case where $D$ is a DVR, $K$ is its quotient field, and $R$ is a non-Eichler $D$-order in $M_2(K)$.
The main work lies in showing that $J(R)\setminus J(R)^2$ contains an element $z$ with $\nr(z)=0$ (see \cref{t-noneichler}); from this result the factorization theoretic properties then follow easily.
To show the existence of such an element $z$, we make use of the fact that every such order $R$ arises as the even Clifford algebra $C_0(M,q)$ of a ternary quadratic module $(M,q)$ over $D$.

A good reference for this result is Chapter 22 of Voight's book \cite{Voight18}.
A correspondence between primitive ternary quadratic forms and Gorenstein orders was established by Brzeziński in \cite{Brzezinski82}.
The specific correspondence used here appears in \cite{Gross-Lucianovic09} and \cite{Voight11}.

Let $D$ be a commutative ring and let $M$ be a $D$-module.
A \emph{quadratic form} is a map $q\colon M \to D$ satisfying the following:
\begin{enumerate}
\item [(i)] $q(dx) = d^2q(x)$ for all $m \in M$, $d \in D$.
\item [(ii)] The map $B \colon M \times M \to D$ defined by
  \[
    B(x,y) = q(x+y) - q(x) - q(y)
  \]
  is symmetric and $D$-bilinear.
\end{enumerate}
The pair $(M,q)$ is referred to as a \emph{quadratic module} (over $D$).

Let $M$ be free of rank $3$ with basis $e_1$, $e_2$, $e_3$, and equipped with the quadratic form
\[
q(xe_1 + ye_2 + ze_3) = ax^2 + by^2 + cz^2 + uyz + vxz + wxy \quad\text{for $x$, $y$, $z \in D$}
\]
where $a$, $b$, $c$, $u$, $v$, $w \in D$.
We refer to $(M,q)$ as a \emph{ternary quadratic module}.
Its \emph{\textup{(}half-\textup{)}discriminant} is
\[
  d'(q) = 4abc + uvw - au^2 - bv^2 -cw^2.
\]
The quadratic module $(M,q)$ is \emph{nondegenerate} if $d'(q) \ne 0$.

The \emph{Clifford algebra} $C(M,q)$ of $(M,q)$ is the quotient of the tensor algebra $T(M)$ by the ideal generated by $x \otimes x - q(x)$ for $x\in M$.
It has a basis consisting of elements $e_{i_1}\cdots e_{i_r}$ with $r \in [0,3]$ and $1 \le i_1 < \cdots < i_r \le 3$.
The Clifford algebra has a $\bZ/2 \bZ$-grading and the basis elements $e_{i_1} \cdots e_{i_r}$ have degrees corresponding to the parity of their length.
The set of elements with even grading form a subalgebra $C_0(M,q)$, the \emph{even Clifford algebra} of $(M,q)$.

The algebra $C_0(M,q)$ is an (associative, unital) $D$-algebra with basis
\[
1,\quad {\quat i}=e_2e_3,\quad {\quat j}=e_3e_1,\quad {\quat k}=e_1e_2
\]
and relations
\begin{equation} \label[pluralequation]{e-relc0}
\begin{aligned}
\quat i^2 &= u\quat i - bc,   & \quat j\quat k &= a\overline{\quat i} = a(u-\quat i), & \quat k\quat j &= -vw + a\quat i + w\quat j + v\quat k,\\
\quat j^2 &= v\quat j - ac,   & \quat k\quat i &= b\overline{\quat j} = b(v-\quat j), & \quat i\quat k &= -uw + w\quat i + b\quat j + u\quat k,\\
\quat k^2 &= w\quat k - ab,   & \quat i\quat j &= c\overline{\quat k} = c(w-\quat k), & \quat j\quat i &= -uv + v\quat i + u\quat j + c\quat k.
\end{aligned}
\end{equation}
This algebra has a standard involution given by $\overline{\quat i} = u - \quat i$, $\overline{\quat j} = v - \quat j$, and $\overline{\quat k} = w - \quat k$.
The reduced norm is $\nr(x)=x\overline x=\overline{x} x$ and the reduced trace is $\tr(x) = x + \overline x$ for $x\in C_0(M,q)$.
The reduced norm is a quadratic form on $C_0(M,q)$.
Its associated bilinear form is $B(x,y) = \nr(x+y) - \nr(x) - \nr(y) = \tr(x\overline{y}) = \tr(y\overline{x})$ for $x$,~$y \in C_0(M,q)$.
In particular, the reduced norm and trace are given by
\[
\begin{split}
\nr(x_0 + x_1 \quat i + x_2 \quat j+ x_3 \quat k) &= x_0^2 + bc x_1^2 + ac x_2^2 + ab x_3^2
                               + u x_0x_1 + v x_0x_2 + w x_0x_3 \\
                               &+ (uv-cw)x_1x_2 + (uw-bv)x_1x_3 + (vw-au)x_2x_3.
\end{split}
\]
and
\[
\tr(x_0 + x_1 \quat i + x_2 \quat j + x_3 \quat k) = 2x_0 + ux_1 + vx_2 + wx_3.
\]

We also note that for any $x\in C_0(M,q)$, \begin{equation}\label{e-charpoly} x^2-\tr(x)x+\nr(x)=0.\end{equation}

We now gather several results in this setting, many of which are well-known at least in some restricted settings.

\begin{lemma} \label{l-nilpotent}
  Let $(M,q)$ be a ternary quadratic module over a commutative domain $D$.
  For $x \in C_0(M,q)$ the following statements are equivalent.
  \begin{equivenumerate}
    \item\label{l-nilpotent:nrtr} $\nr(x)=\tr(x)=0$.
    \item\label{l-nilpotent:2} $x^2 = 0$.
    \item\label{l-nilpotent:n} $x$ is nilpotent.
  \end{equivenumerate}
\end{lemma}

\begin{proof}
  \ref*{l-nilpotent:nrtr}${}\Rightarrow{}$\ref*{l-nilpotent:2}:
  By \cref{e-charpoly}, if $\tr(x)=\nr(x)=0$, then $x^2=0$.

  \ref*{l-nilpotent:2}${}\Rightarrow{}$\ref*{l-nilpotent:n}: Trivial.

  \ref*{l-nilpotent:n}${}\Rightarrow{}$\ref*{l-nilpotent:nrtr}:
  Since $x$ is a zero-divisor in $C_0(M,q)$ and $C_0(M,q)$ is a torsion-free $D$-module, we must have $x\overline{x}=\overline x x=\nr(x)=0 \in D$.
It follows that $x^2 = \tr(x) x$ and hence $x^n = \tr(x)^{n-1} x$ for all $n \ge 1$.
  Since $x$ is nilpotent, there exists an $n \ge 1$ such that $\tr(x)^{n-1}x=0$.
  Since $C_0(M,q)$ is a torsion-free $D$-module, this implies $\tr(x) = 0$ or $x=0$, which also implies $\tr(x) = 0$.
\end{proof}

If $(M,q)$ is a quadratic module and $B$ is the bilinear form associated to $q$, then $M^\perp = \{\, x \in M : B(x,y) = 0 \,\}$, and the \emph{quadratic radical} of $(M,q)$ is
\[
  \rad M = \rad\, (M,q) = \{\, x \in M^\perp : \nr(x) = 0 \,\}.
\]
If $2 \in D^\bullet$, then $2\nr(x)=B(x,x)$ implies $\rad M = M^\perp$.

\begin{lemma} \label{l-radc0}
  Let $(V,q)$ be a ternary quadratic module over a field $K$.
  For $A=C_0(V,q)$ we have
  \[
  J(A) = \rad\, (A,\nr).
  \]
  If $\chr K \ne 2$, then $J(A) = A^\perp$.
\end{lemma}

\begin{proof}
  Since $A$ is a $4$-dimensional $K$-algebra, it is Artinian and hence the Jacobson radical $J(A)$ is nilpotent.

  Let $x \in J(A)$.
  Then $x\overline y \in J(A)$ is nilpotent for all $y \in A$, and hence $B(x,y)=\tr(x\overline y)=0$.
  Thus $x \in A^\perp$.
  Moreover, since $x$ is nilpotent, $\nr(x) = 0$.

  Let $x \in A^\perp$ with $\nr(x) = 0$.
  We show that $xA$ is a nil right ideal, that is, that $xy$ is nilpotent for all $y \in A$.
  Let $y \in A$.
  Then $\nr(xy) = 0$ and $\tr(xy) = \tr(x \overline{\overline{y}}) = B(x,\overline y) = 0$.
  We conclude that $xy$ is nilpotent.

  If $\chr K \ne 2$, then $A^\perp = \rad\, (A,\nr)$.
\end{proof}

If $(V_1,q_1)$ and $(V_2,q_2)$ are quadratic modules, then $(V_1 \perp V_2)^\perp = V_1^\perp \perp V_2^\perp$.
Using this fact along with the following basic result (\cref{l-radform}), $V^\perp$ of a quadratic module over a field $K$ can easily be computed from an orthogonal decomposition of $V$.
We use this in \cref{p-radres}.

\begin{lemma} \label{l-radform}
  Let $(V,q)$ be a quadratic module over a field $K$.
  \begin{enumerate}
  \item
    Let $V=Ke_1$ be $1$-dimensional and $q(xe_1) = ax^2$ for some $a \in K$.
    Then
    \[
    V^\perp =
    \begin{cases}
      \mathbf 0 &\text{if $a \ne 0$ and $\chr K \ne 2$,} \\
      V         &\text{if $a = 0 $ or $\chr K = 2$.}
    \end{cases}
    \]
  \item If $\chr K = 2$, $V=Ke_1\oplus Ke_2$, and $q(xe_1 + ye_2) = ax^2 + u xy + by^2$ with $a$, $u$, $b \in K$, then $V^\perp = \mathbf 0$ if $u \in K^\times$ and $V^\perp = V$ if $u=0$.
  \end{enumerate}
\end{lemma}

\begin{proposition} \label{p-radres}
  Let $(V,q)$ be a ternary quadratic module over a field $K$, and let $e_1$, $e_2$,~$e_3$ be a basis of $V$.
  Let $A = C_0(V,q)$.

  \begin{enumerate}
    \item\label{p-radres:abc} Suppose $q(xe_1+ye_2+ze_3)=ax^2 + by^2 + cz^2$ with $a$, $b$,~$c \in K^\times$.
      \begin{enumerate}
        \item If $\chr K \ne 2$, then $q$ is nondegenerate, $J(A)=J(A)^2 = \mathbf 0$, and $A=A/J(A)$ is a quaternion algebra over $K$.
        \item Suppose $\chr K = 2$.
          Then $q$ is degenerate.
          \begin{enumerate}
          \item\label{p-radres:abc:1}
            If $bc=y_0^2$ and  $ac=z_0^2$ with $y_0$,~$z_0 \in  K^\times$, then
            \begin{align*}
              J(A)\phantom{^1}  &=\langle y_0+\quat i,\ z_0+\quat j\rangle_A = \langle y_0 + \quat i,\ z_0 + \quat j,\ y_0z_0 + c\quat k \rangle_K,\\
              J(A)^2 &= \langle y_0z_0 + z_0 \quat i + y_0 \quat j + c \quat k \rangle_K, \quad\text{and}\\
              J(A)^3 &= \mathbf 0.
            \end{align*}
          \item\label{p-radres:abc:2} If $ac$, $bc$ are not both squares in $K$ and there exist $y_0$, $y_1$,~$y_2 \in K$ with $bc y_1^2 + acy_2^2=y_0^2$ and $(y_1,y_2) \ne (0,0)$ \textup{(}equivalently, $q$ is isotropic\textup{)}, then
            \[
            J(A) = \langle y_0+y_1\quat i + y_2\quat j \rangle_A = \langle y_0+y_1\quat i + y_2\quat j,\  y_1bc + y_0\quat i + y_2c\quat k \rangle_K,
            \]
            and $J(A)^2=\mathbf 0$.
          \item\label{p-radres:abc:3} If there exist no $y_0$, $y_1$,~$y_2 \in K$ with $bc y_1^2 + acy_2^2=y_0^2$ and $(y_1,y_2)\ne (0,0)$ \textup{(}equivalently, $q$ is anisotropic\textup{)}, then $J(A) = \mathbf 0$.
          \end{enumerate}
          In each of these three cases, $A/J(A) \cong K(\sqrt{bc}, \sqrt{\vphantom{b}ac})$.
      \end{enumerate}

    \item\label{p-radres:ab} Suppose $q(xe_1+ye_2+ze_3)=ax^2 + by^2$ with $a$,~$b \in K^\times$.
      Then $q$ is degenerate.
      \begin{enumerate}
        \item If $\chr K \ne 2$, then $J(A) = \langle \quat i,\quat j\rangle_K$ and $J(A)^2=\mathbf 0$.
          We have
          \[
          A/J(A) \cong K[X]/\langle X^2+ab \rangle.
          \]
          Thus $A/J(A) \cong K \oplus K$ if $-ab \in (K^\times)^2$ and $A/J(A) \cong K(\sqrt{-ab})$ otherwise.

        \item Suppose $\chr K=2$.
          \begin{enumerate}
            \item\label{p-radres:ab:1} If $ab =y_0^2$ with $y_0 \in K^\times$, then $J(A) = \langle \quat i, \quat j,\  y_0 + \quat k \rangle_K$ and $J(A)^2 = \langle y_0\quat i+b\quat j \rangle_K=\langle a \quat i+y_0 \quat j \rangle_K$.
              Finally, $J(A)^3 = \mathbf 0$.
            \item\label{p-radres:ab:2} If $ab \in K^\times \setminus (K^\times)^2$, then $J(A) = \langle \quat i, \quat j \rangle_K$ and $J(A)^2=\mathbf 0$.
          \end{enumerate}
          In both cases, $A/J(A) \cong K(\sqrt{ab}) = K(\sqrt{-ab})$.
      \end{enumerate}

    \item\label{p-radres:a} Suppose $q(xe_1+ye_2+ze_3)=ax^2$ with $a \in K$.
      Then $q$ is degenerate, $J(A) = \langle \quat i,\quat j,\quat k \rangle_K$, and $A/J(A) \cong K$.
      We have $J(A)^2=\langle \quat i \rangle_K$ if $a \in K^\times$ and $J(A)^2=\mathbf 0$ if $a=0$.
      In either case, $J(A)^3 = \mathbf 0$.

    \item\label{p-radres:abcu} Suppose $\chr K = 2$ and $q(xe_1+ye_2+ze_3) = ax^2 + by^2 + cz^2 + uyz$ with $a$,~$u \in K^\times$ and $b$,~$c \in K$.
      Then $q$ is nondegenerate, $J(A) = \mathbf 0$, and $A=A/J(A)$ is a quaternion algebra.

    \item\label{p-radres:bcu} Suppose $\chr K = 2$ and $q(xe_1+ye_2+ze_3) = by^2 + cz^2 + uyz$ with $u \in K^\times$ and $b$, $c \in K$.
      Then $q$ is degenerate, $J(A) = \langle \quat j,\quat k\rangle_K$, and $J(A)^2=\mathbf 0$.
      We have
      \[
      A/J(A) \cong K[X]/\langle X^2  + uX + bc\rangle.
    \]
    Thus $A/J(A) \cong K \oplus K$ if $bc=y_0^2+uy_0$ for some $y_0 \in K$, and $A/J(A)$ is a quadratic separable field extension of $K$ otherwise.
  \end{enumerate}
\end{proposition}

\begin{proof}
  We first consider \labelcref*{p-radres:abc,p-radres:ab,p-radres:a} in case $\chr K \ne 2$.
  Here $q(xe_1 + ye_2 + ze_3) = ax^2+by^2+cz^2$ with $a$,~$b$,~$c \in K$ and $\nr(x_0 + x_1\quat i+ x_2\quat j + x_3\quat k) = x_0^2 + bcx_1^2 + acx_2^2 + abx_3^2$.
  Since $\chr K \ne 2$, \cref{l-radc0} implies $J(A)=\rad\, (A,\nr) = A^\perp$, which is straightforward to compute.
  The claims about $J(A)^2$ and $A/J(A)$ then follow using \cref{e-relc0}.

  Now suppose $\chr K = 2$.
  We again first consider \labelcref*{p-radres:abc,p-radres:ab,p-radres:a}, that is, $q(xe_1 + ye_2 + ze_3) = ax^2+by^2+cz^2$ with $a$,~$b$,~$c \in K$ is diagonal.
  Inspection of \cref{e-relc0} shows that, due to $\chr K = 2$, the algebra $A$ is commutative.

  \ref*{p-radres:abc}:
  If $abc\ne 0$, the homomorphism $K[X,Y] \to A$ that maps $X \mapsto \quat i$ and $Y \mapsto \quat j$ induces an isomorphism
  \[
  A \cong K[X,Y]/\langle X^2+bc, Y^2+ac \rangle \cong K[X]/\langle X^2+bc\rangle  \otimes_K K[Y]/\langle Y^2+ac\rangle.
  \]

  \ref*{p-radres:abc:1}:
  If $bc=y_0^2$ and $ac=z_0^2$ with $(y_0,z_0) \in K$, then $K[X]/\langle X^2+bc\rangle \cong K[X]/\langle X+y_0\rangle^2$ and $K[Y]/\langle Y^2+ac\rangle \cong K[Y]/\langle Y+z_0\rangle^2$.
  Thus
  \[
  A \cong K[X,Y]/\langle (X+y_0)^2,(Y+z_0)^2\rangle
  \]
  Since $A$ is commutative and Artinian, $J(A)$ is equal to the nilradical of $A$.
  Thus $J(A)$ is generated by $\quat i+y_0$ and $\quat j+z_0$.
  It has a $K$-basis given by $\quat i+y_0$, $\quat j+z_0$ and $(\quat i+y_0)(\quat j+z_0)=y_0z_0+z_0\quat i+y_0\quat j + c\quat k$, which can easily be transformed into the claimed basis.
  The claims about $J(A)^2$ and $J(A)^3$ follow by direct computation.

  \labelcref*{p-radres:abc:2,p-radres:abc:3}:
  Note first that the condition in \ref*{p-radres:abc:3} implies that $ab$ and $bc$ are non-squares in $K$.
  Moreover, note that, if $bc$ is a non-square, then $bc y_1^2 + acy_2^2 =y_0^2$ has a solution with $(y_1,y_2) \ne (0,0)$ if and only if $ac$ is a square in $K(\sqrt{bc})$.
  (This follows because the squares of $K(\sqrt{bc})$ are precisely the elements of the form $z_0^2+bcz_1^2=(z_0+\sqrt{bc}z_1)^2$ with $z_0$,~$z_1 \in K$.)
  Analogously, if $ac$ is a non-square, then $bc y_1^2 + acy_2^2 =y_0^2$ has a solution with $(y_1,y_2) \ne (0,0)$ if and only if $bc$ is a square in $K(\sqrt{ac})$.

  Suppose now that $bc$ is not a square in $K$.
  Then $K[X]/\langle X^2+bc\rangle \cong K(\sqrt{bc})$ is a (purely inseparable) quadratic field extension of $K$ and
  \[
  A \cong K(\sqrt{bc})[Y]/\langle Y^2 + ac\rangle.
  \]

  In case \ref*{p-radres:abc:2}, the element $ac$ is a square in $K(\sqrt{bc})$; explicitly
  \[
  \sqrt{ac}= y_2^{-1}y_0 + y_2^{-1}y_1 \sqrt{bc}.
  \]
  Thus, $J\big(K(\sqrt{bc})[Y]/\langle Y^2 + ac\rangle\big)$ is generated by the residue class of $Y+y_2^{-1}y_0+y_2^{-1}y_1\sqrt{bc}$.
  It follows that $J(A)$ is generated by $y=y_0 + y_1 \quat i + y_2 \quat j$; a $K$-basis is given by $y$, $\quat i y$.
  We also note that $y^2=0$, hence $J(A)^2=\mathbf 0$.
  In case \ref*{p-radres:abc:3}, the ring $A$ is a biquadratic field extension of $K$ and $J(A) = \mathbf 0$.

  \ref*{p-radres:ab}:
  In this case $\nr(x_0 + x_1\quat i + x_2\quat j + x_3\quat k) = x_0^2 + abx_3^2$ and hence $A^\perp=A$.
  Since $J(A) = \{\, x \in A^\perp : \nr(x) = 0 \,\}$ we have $\langle \quat i,\quat j\rangle_K = \langle \quat i,\quat j \rangle_A \subset J(A)$.
  \cref{e-relc0} imply
  \[
  A / \langle \quat i,\quat j \rangle \cong K[X]/\langle X^2+ab\rangle,
  \]
  where $\quat k$ is mapped to $X$.
  The claims follows similarly to the case $abc\ne 0$.

  \ref*{p-radres:a}:
  Here we have $\nr(x_0 + x_1\quat i + x_2\quat j + x_3\quat k) = x_0^2$ and hence $J(A) = \langle \quat i,\quat j,\quat k \rangle_K$ and $A/J(A) \cong K$.

  Finally we consider \labelcref*{p-radres:abcu,p-radres:bcu} where $\nr(x_0 + x_1\quat i + x_2\quat j + x_3\quat k) = (x_0^2 + ux_0x_1 + bcx_1^2) + a(cx_2^2 - u x_2x_3 + bx_3^2)$.

  \labelcref*{p-radres:abcu}:
  If $a \in K^\times$, then $d'(q) = -au^2 \ne 0$.
  Thus $q$ is nondegenerate.
  The form $q$ is similar to $x^2 + b' y^2 + c' z^2 + yz$, and hence $A$ is easily recognized as a quaternion algebra from \cref{e-relc0}.
  It follows that $J(A)=\mathbf 0$.

  \labelcref*{p-radres:bcu}:
  Since $a = 0$, we have $A^\perp = \langle \quat j,\quat k \rangle_K$.
  Since $\nr$ vanishes on all of $A^\perp$, it follows that $J(A) = \langle \quat j,\quat k\rangle_K$.
  From \cref{e-relc0} one deduces $J(A)^2=\mathbf 0$ as well as the stated form of $A/J(A)$.
\end{proof}

\begin{remark}
  If $\chr K=2$ and $K$ is perfect, then every element in $K$ is a square, and hence the dyadic case simplifies considerably.
  In our later applications, this will correspond to the assumption that all residue fields of characteristic $2$ are perfect.
\end{remark}

Let $D$ be a DVR, let $K$ be its quotient field, and let $R$ be an order in a quaternion algebra over $K$.
Let $M$ be a free $D$-module of rank $3$ with basis $e_1$, $e_2$,~$e_3$.
There exists a quadratic form $q\colon M \to D$ with $d'(q) \ne 0$ such that $R\cong C_0(M,q)$.
(This follows from \cite[Theorem B]{Voight11} by specializing to the case of DVRs; alternative see \cite[Main Theorem 22.4.1 or Proposition 22.4.2]{Voight18}.)
Since similar forms, and thus in particular isometric forms, give rise to isomorphic orders, we may assume that either
\[
q(xe_1 + ye_2 + ze_3) = ax^2 + by^2 + cz^2
\]
with $\val(a) \le \val(b) \le \val(c)$ and $4abc \ne 0$, or $\chr D/\pi D = 2$ and
\[
q(xe_1 + ye_2 + ze_3) = ax^2 + by^2 + cz^2 + uyz
\]
with $a(4bc - u^2)\ne 0$ and $\val(u) < \val(2b) \le \val(2c)$.
(See \cite[Proposition 3.10]{Voight13} or \cite[(15.1)]{Kneser02}.)

Using this and the previous proposition, we make some preliminary structural observations about quaternion orders.

\begin{corollary} \label{c-resprop}
  Let $D$ be a DVR, let $K$ be its quotient field, and let $R$ be an order in a quaternion algebra over $K$.
  \begin{enumerate}
    \item \label{c-resprop:dim} $R/J(R)$ is a finite-dimensional $D/\pi D$-algebra and $\dim_{D/ \pi D} R/J(R) \in \{1,2,4\}$.
    \item \label{c-resprop:max} If $\dim_{D/\pi D} R/J(R) = 4$ \textup{(}equivalently, $J(R/\pi R)=\mathbf 0$\textup{)}, then $R$ is a maximal order.
    \item \label{c-resprop:comm} If $R$ is not a maximal order, then $R/J(R)$ is commutative.
      In particular, $J(R/J(R))=\sqrt{R/J(R)}$.
    \item \label{c-resprop:local} If $R$ is not an Eichler order, then $R$ is a local ring \textup{(}that is, $R/J(R)$ is a division ring\textup{)}.
  \end{enumerate}
\end{corollary}

\begin{proof}
  Let $\res = D/\pi D$ and let $B=R/\pi R$.
  Then $\pi R \subset J(R)$ and $R/J(R) \cong B/J(B)$ (see \cite[Theorem 6.15]{Reiner}).
  In particular, $R/J(R)$ is a finite-dimensional $\res$-algebra.
  Since $R \cong C_0(M,q)$ we find that $B \cong C_0(M/\pi M, \overline{q})$, where $\overline q$ is the reduction of the quadratic form $q$ modulo $\pi D$.
  Thus we can apply the results of \cref{p-radres} to the $\res$-algebra $B$.

  \ref*{c-resprop:dim}
  The claim about the dimensions follows by inspection of the cases in \cref{p-radres}.

  \ref*{c-resprop:max}
  The inclusion $\pi R \subset J(R)$ together with $\dim_\res R/J(R)=\dim_\res R/\pi R = 4$ implies $J(R)=\pi R$.
  Therefore $J(R)$ is invertible and hence $R$ is hereditary by \cite[Theorem 39.1]{Reiner}.
  The only non-maximal hereditary orders are the non-maximal Eichler orders, where $R/J(R) \cong \res \oplus \res$.
  Thus $R$ must be maximal.

  \ref*{c-resprop:comm}
  If $R$ is not maximal, we must have $\dim_\res R/J(R) \in \{1,2\}$ by \ref*{c-resprop:dim} and \ref*{c-resprop:max}.
  Thus $R/J(R)$ is commutative.
  (Alternatively, this also follows by inspection from \cref{p-radres}.)
  Since $R/J(R)$ is commutative and Artinian, the Jacobson radical and the nilradical coincide.

  \ref*{c-resprop:local}
  If $R$ is not an Eichler order, it is in particular not maximal, and hence $\dim_{\res} R/J(R) \in \{1,2\}$.
  Thus $R/J(R)$ is isomorphic to  $\res $, to $\res[X]/\langle X^2\rangle$, to $\res \oplus \res$, or to  a quadratic field extension of $\res$.
  But $R/J(R) \cong \res[X]/\langle X^2 \rangle$ is impossible since $\sqrt{R/J(R)} =\mathbf 0$, and $R/J(R)\cong \res \oplus \res$ would imply that $R$ is an Eichler order (see \cite[Proposition 2.1]{Brzezinski83}).
  Thus $R/J(R)$ is a field.
\end{proof}

\begin{proposition} \label{p-quatmat}
  Let $(M,q)$ be a ternary quadratic module over a commutative domain $D$.
  Suppose $d'(q) \ne 0$, let $R=C_0(M,q)$, and let $A = K \otimes_D C_0(M,q)$ be its quotient ring.
  The following statements are equivalent:
  \begin{equivenumerate}
    \item \label{p-quatmat:mat}$A \cong M_2(K)$.
    \item \label{p-quatmat:isotropic}$(R, \nr)$ is isotropic.
    \item \label{p-quatmat:0isotropic} $(R_0, \nr|_{R_0})$ is isotropic, where $R_0 = \{\, x \in R : \tr(x) = 0 \,\} = \langle 1 \rangle^\perp$.
    \item \label{p-quatmat:qisotropic} $(M,q)$ is isotropic.
  \end{equivenumerate}
\end{proposition}

\begin{proof}
  Since $q$ is nondegenerate, the ring $R$ is a quaternion order in the quaternion algebra $A$ (see \cite[Theorem 22.3.1]{Voight18}).
  
  \ref*{p-quatmat:mat}$\ \Rightarrow\ $\ref*{p-quatmat:0isotropic}$\ \Rightarrow \ $\ref*{p-quatmat:isotropic} is clear since $\nr(x)=\det(x)$ for $x \in A$ (a consequence of the uniqueness of the standard involution).

  \ref*{p-quatmat:isotropic}$\ \Rightarrow\ $\ref*{p-quatmat:mat}:
  Since $A$ is a quaternion algebra, it must be either a division ring or isomorphic to $M_2(K)$.
  It cannot be the former, because any element $x \in R$ with $\nr(x)=0$ is a zero-divisor.

  \ref*{p-quatmat:0isotropic}$\ \Leftrightarrow \ $\ref*{p-quatmat:qisotropic} holds since $(R_0, d'(q)\nr|_{R_0})$ and $(M,q)$ are isometric by \cite[(6.20)]{Kneser02}.
\end{proof}

With all necessary preliminary results laid out, we now arrive at the main result of this section.

\begin{theorem}\label{t-noneichler}
  Let $D$ be a DVR with quotient field $K$, and $R$ a quaternion $D$-order in $M_2(K)$.
  \begin{enumerate}
    \item\label{t-noneichler:nilpotent} There exists $z \in J(R) \setminus J(R)^2$ with $\nr(z)=0$.
    \item\label{t-noneichler:rho} If $R$ is not an Eichler order \footnote{If $R$ is a non-maximal Eichler order, then the same holds by \cref{l-eichler-longatom}. If $R$ is a maximal order, then $R \cong M_2(D)$. In this case $z=\begin{bsmallmatrix} 0 & \pi \\ 0 & 0 \end{bsmallmatrix}$ satisfies $z \in J(R) \setminus J(R)^2$.},
      then there exists an $N \in \bN_0$, such that for every $n \ge N$, there exists an atom $u \in R^\bullet$ with $\val(\nr(u))=2n$.
  \end{enumerate}
\end{theorem}

\begin{proof}
  \ref*{t-noneichler:nilpotent}
  Let $M$ be a free $D$-module of rank $3$ with basis $e_1$, $e_2$,~$e_3$.
  There exists a quadratic form $q\colon M \to D$ with $d'(q) \ne 0$ such that $R=C_0(M,q)$.
  Since isometric forms give rise to isomorphic orders, we may again assume that either
  \[
  q(xe_1 + ye_2 + ze_3) = ax^2 + by^2 + cz^2
  \]
  with $\val(a) \le \val(b) \le \val(c)$ and $4abc \ne 0$, or $\chr D/\pi D = 2$ and
  \[
  q(xe_1 + ye_2 + ze_3) = ax^2 + by^2 + cz^2 + uyz
  \]
  with $a(4bc-u^2)\ne 0$ and $\val(u) < \val(2b) \le \val(2c)$.
  By assumption, $K \otimes_D R \cong M_2(K)$, and thus the quadratic module $(M \otimes_D K,q)$ is isotropic.
  By clearing denominators, we see that also $(M,q)$ is isotropic.

  We write $q$ as $q(xe_1 + ye_2 + ze_3) = ax^2 + by^2 + cz^2 + uyz$ (possibly with $u=0$).
  Since $q$ is isotropic, the same is true for $aq(x e_1 + y e_2 + z e_3) = (ax)^2 + a(by^2 + cz^2 + uyz)$.
  Thus there exist $z_0$, $z_2$,~$z_3 \in D$ such that $z_0^2 + a(b z_3^2 + c z_2^2 + uz_2z_3)=0$.
  We may also assume $\min\{\val(z_0), \val(z_2), \val(z_3) \} = 0$.
  Since $\val(z_0)=0$ and $\val(z_2)$,~$\val(z_3) > 0$ would lead to a contradiction, we must have $\min\{\val(z_2), \val(z_3)\} = 0$.
  Set $z=z_0 + z_2 \quat j - z_3 \quat k \in R$, so that $\nr(z) = 0$.
  This implies that $z$ is a zero-divisor in $R$.
  Since $R$ is not an Eichler order, it is a local ring, and hence $z \in J(R)$.

  Note that $R/\pi R \cong C_0(M/\pi M, \overline q)$ where $\overline q\colon M/\pi M \to D/\pi D$ is the reduction of $q$ modulo $\pi D$.
  Since the $\quat i$-coordinate of $z$ is zero while $\min\{\val(z_2), \val(z_3)\} = 0$, inspection of the individual cases for $\overline q$ in \cref{p-radres} shows $z+ \pi R \not \in J(R/\pi R)^2$.
  Hence $z \not \in J(R)^2$.

  \ref*{t-noneichler:rho}
  The classification of orders above implies that any non-Eichler order is a local ring.
  It follows that $J(R)^\bullet \setminus J(R)^2 \subset \cA(R^\bullet)$.
  Let $z$ be as in \ref*{t-noneichler:nilpotent}.
  Then $\nr(\pi^k + z) =\nr(\pi^k) + \nr(z) + \tr(\pi^k z) = \pi^{2k} + \pi^k\tr(z)$ for all $k \in \bN_0$.
  For $k \ge 2$, we have $\pi^k \in J(R)^2$ and hence $\pi^k + z \in J(R) \setminus J(R)^2$.
  Therefore $\pi^k+z$ is an atom of $R^\bullet$.
  If $\tr(z)=0$, then $\val(\nr(\pi^k + z)) = 2k$.
  Otherwise, $\val(\tr(z))$ is constant, and, for large enough $k$, we find $\val(\nr(\pi^k+z)) = k + \val(\tr(z))$.
  In particular, by suitable choice of $k$, we can realize any sufficiently large even value for $\val(\nr(\pi^k+z))$.
\end{proof}

\begin{corollary} \label{c-m2-minlen}
  If $R \subset M_2(K)$ is a non-Eichler order and $a \in R^\bullet$, there exists an $M \in \bN_0$ such that $\min\sL(a) \le M$ for all $a \in R^\bullet$.
\end{corollary}

\begin{proof}
  By \subref{l-pipower:2}, there exists an $N_1 \in \bN_0$ such that every $x \in R^\bullet$ can be written in the form $a=\pi^{2m+d} \varepsilon u_1\cdots u_n$ with $n < N_1$, $m \in \bN_0$, $d \in \{0,1\}$, $\varepsilon \in R^\times$, and $u_1$, $\ldots\,$,~$u_n \in \cA(R^\bullet)$.

  Let $N_2 \in \bN_0$ be the constant from \subref{t-noneichler:rho}.
  If $m \ge N_2$, then there exists $u \in \cA(R^\bullet)$ such that $\val(\nr(u))=2m$, and hence $\pi^{2m}=u\overline{u} \delta$ for some $\delta \in D^\times$.
  Then
  \[
    a = u \overline{u} \delta\varepsilon u_1\cdots u_n
  \]
  shows $\min \sL(a) \le 2+n < 2+N_1$.
  If $m < N_2$, then $\min \sL(\pi) \le 2$ implies $\min\sL(a) \le 2(2m+d) + n < 4N_2 + 2 + N_1$.
  Thus the claim holds with $M=4N_2 + 1 + N_1$.
\end{proof}

From \cref{t-noneichler} and \cref{c-m2-minlen} it is now easy to deduce the factorization-theoretic properties for non-Eichler orders in $M_2(K)$ that we claim in the introduction.
We do so in a more general setting in the next section.

\section{Putting everything together: Non-complete case}
\label{s-wrapup}

We have so far finished the proof of \cref{t-main} in the case of orders for which $\widehat A$ is a division ring (in \cref{s-division}).
We have also treated all non-hereditary orders in $M_2(K)$ in \cref{s-eichler,s-noneichler} (but not yet derived the final results in these cases).
What is missing is the case where $A$ is a division ring, but $\widehat A \cong M_2(\widehat K)$.

To wrap everything up, and to cover this last case, we will first show in this section that the properties in \cref{t-lengths,l-mindelta,c-m2-minlen} carry over from $\widehat R$ to $R$.
In particular, we show that there is a transfer homomorphism from $R^\bullet$ to $\widehat{R}^\bullet$.
We then show how to derive \cref{t-main} from these properties.
We begin with a more general result about isoatomic transfer homomorphisms.

\begin{proposition} \label{p-strongtransfer}
  Let $H$, $T$ be monoids and let $\varphi \colon H \to T$ be a homomorphism such that
  \begin{itemize}
    \item $\varphi$ is a left divisor homomorphism \textup{(}that is, whenever $a$, $b \in H$ are such that $\varphi(b) \in \varphi(a)T$, then $b \in aH$\textup{)},
    \item $\varphi(H) T^\times = T$.
  \end{itemize}
  Then:
  \begin{enumerate}
  \item\label{p-strongtransfer:transfer} $\varphi$ is an isoatomic transfer homomorphism.
  \item\label{p-strongtransfer:fact} There exists a unique homomorphism $\varphi^*\colon \sZ^*(H) \to \sZ^*(T)$ satisfying
    \[
    \varphi^*(u) = \varphi(u)\quad\text{and}\quad \varphi^*(\varepsilon) = \varphi(\varepsilon)
    \]
    for all $u \in \cA(H)$ and $\varepsilon \in H^\times$.
    We have a commutative diagram
    \[
    \begin{tikzcd}
      \sZ^*(H) \ar[r,"\varphi^*"] \ar[d,"\pi_H"] & \sZ^*(T) \ar[d,"\pi_T"] \\
      H \ar[r,"\varphi"] & T.
    \end{tikzcd}
    \]
  \item\label{p-strongtransfer:injective} For every $a \in H$, the map $\varphi^*|_{\sZ_H^*(a)} \colon \sZ_H^*(a) \to \sZ_T^*(\varphi(a))$ is a bijection and $\varphi^*(\sZ^*(H)) = \sZ^*(T) T^\times$.
  \end{enumerate}
\end{proposition}

\begin{proof}
  \ref*{p-strongtransfer:transfer}
  We must show the following:
  \begin{itemize}
    \item $\varphi^{-1}(T^\times) \subset H^\times$.
    \item If $a \in H$ and if there are $s$,~$t \in T$ such that $\varphi(a)=st$, then there exist $b$,~$c \in H$ and $\varepsilon \in T^\times$ such that $a=bc$, $\varphi(b)=s\varepsilon^{-1}$, and $\varphi(c)=\varepsilon t$.
    \item If $u$,~$v \in \cA(H)$ are such that $\varphi(u)$ and $\varphi(v)$ are associated in $T$, then $u$ and $v$ are associated in $H$.
  \end{itemize}

  Let $a \in H$ with $\varphi(a) \in T^\times$.
  Then $1 \in \varphi(a)T$, and hence $1 \in aH$, and $a$ is right invertible.
  Since $H$ is cancellative, $a \in H^\times$.

  Now let $a \in H$ and let $s$,~$t \in T$ be such that $\varphi(a)=st$.
  By assumption, there exist $b \in H$ and $\varepsilon \in T^\times$ such that $s=\varphi(b)\varepsilon$.
  Then $\varphi(a) \in \varphi(b)T$ and so $a \in bH$.
  Let $c \in H$ be such that $a=bc$.
  Then $\varphi(a)=\varphi(b)\varphi(c)=s\varepsilon^{-1}\varphi(c)$.
  Cancellativity in $T$ implies $\varphi(c) = \varepsilon t$.

  If $u$,~$v \in \cA(H)$ with $\varphi(u) = \varphi(v)$  then $\varphi(u) \in \varphi(v)T$ and $\varphi(v) \in \varphi(u) T$.
  Thus $u \in vT$ and $v \in uT$, so that $u \simeq v$.
  (We do not use that $u$,~$v$ are atoms.)

  \ref*{p-strongtransfer:fact}
  Since $\varphi$ is a transfer homomorphism, we have $\varphi(u) \in \cA(T)$ for all $u \in \cA(H)$.
  The claims are now straightforward to check using the construction of the monoid of factorizations.

  \ref*{p-strongtransfer:injective}
  Let $z=\rf[\varepsilon]{u_1,\cdots,u_k}$,~$z'=\rf[\varepsilon']{v_1,\cdots,v_l} \in \sZ_H^*(a)$ with $\varphi^*(z)=\varphi^*(z')$.
  Then $k=l$ and $a=\varepsilon u_1\cdots u_k = \varepsilon' v_1\cdots v_k$.
  If $k=0$, then $\varepsilon=\varepsilon'$ and hence $z=z'$.
  Suppose $k > 1$.
  Then there exist $\delta_2$, $\ldots\,$,~$\delta_{k} \in T^\times$ and $\delta_{k+1}=1$ such that
  \[
  \varphi(\varepsilon' v_1) = \varphi(\varepsilon u_1) \delta_2^{-1} \quad\text{and}\quad \varphi(v_i) = \delta_i \varphi(u_i) \delta_{i+1}^{-1} \quad\text{for all $i \in [2,k]$.}
  \]
  Since $\varphi(\varepsilon'v_1) \in \varphi(\varepsilon u_1) T$, our assumptions also imply $\varepsilon'v_1 \in \varepsilon u_1 H$.
  Thus there exists $\eta_2 \in H^\times$ with $\varepsilon'v_1 = \varepsilon u_1 \eta_2^{-1}$.
  Note that $\delta_2 = \varphi(\eta_2)$.

  We now define $\eta_3$, $\ldots\,$,~$\eta_{k}$ inductively so that $v_i = \eta_i u_i \eta_{i+1}^{-1}$ and $\varphi(\eta_i) = \delta_i$.
  This is possible since $\varphi(\eta_i^{-1}v_i) = \varphi(u_i)\delta_{i+1}^{-1}$ implies that there exists $\eta_{i+1}$ as claimed.
  Finally, since $\varepsilon u_1\cdots u_k = \varepsilon' v_1 \cdots v_k = \varepsilon u_1 \cdots u_{k-1} \eta_{k}^{-1} v_k$, we see $v_k = \eta_{k} u_k$.
  Thus, putting everything together, $z=z'$.

  Now let $y=\rf[\eta]{v_1,\cdots,v_k} \in \sZ^*_T(\varphi(a))$.
  If $k=0$, then necessarily, $a \in H^\times$ and $\varphi^*(a)=\varphi(a)=\eta=y$.
  Suppose $k > 0$.
  By assumption, there exists $\delta_2 \in T^\times$ and $u_1 \in \cA(H)$ such that $\eta v_1=\varphi(u_1) \delta_2^{-1}$.
  Similarly, for $i \in [2,k-1]$ we inductively choose $\delta_{i+1} \in T^\times$ and $u_i \in \cA(H)$ such that $\delta_{i}^{-1}v_i= \varphi(u_i) \delta_{i+1}^{-1}$.
  It follows that $\varphi(a)=\eta v_1 \cdots v_{k-1}v_k = \varphi(u_1\cdots u_{k-1})\delta_k^{-1}v_k$.
  Since $a \in H$ and $u_1\cdots u_{k-1} \in H$, our assumptions imply that there exists $u_k \in H$ with $a = u_1\cdots u_{k-1} u_k$.
  Then, necessarily, $\varphi(u_k) = \delta_k^{-1}v_k$ and so $u_k \in \cA(H)$ and $\varphi^*(\rf{u_1,\cdots,u_k}) = y$.

  To show the final claim, that $\varphi(\sZ^*(H)) T^\times = \sZ^*(T)$, let $y \in \sZ^*(T)$.
  Then $\pi_T(y) \in T$ and hence there exist $a \in H$ and $\varepsilon \in T^\times$ such that $\pi_T(y) = \varphi(a)\varepsilon$.
  Then $y\rfop \varepsilon^{-1} \in \sZ_T^*(\varphi(a))$.
  By what we have already shown, there exists $z \in \sZ_H^*(a)$ such that $\varphi^*(z)=y\rfop \varepsilon^{-1}$.
\end{proof}

We now give one final result before concluding with the proof of \cref{t-main}.

\begin{proposition} \label{p-shortcat}
  Let $H$ be an atomic monoid.
  Suppose that there exists an ideal $J \subset H$ \textup{(}that is, $HJ \subset J$ and $JH \subset J$\textup{)} and an $N \in \bN$ such that:
  \begin{enumerate}
  \item\label{p-shortcat:req1} If $a \in H \setminus J$, then $\card{\sZ^*(a)} = 1$.
  \item\label{p-shortcat:reqsub} If $k \ge N+1$, $u_1$, $\ldots\,$,~$u_k \in \cA(H)$, and $u_1\cdots u_k \in J$, then there exist $l \in [1,N+1]$ and  $i \in [1,k-l+1]$ such that $u_iu_{i+1}\cdots u_{i+l-1} \in J$.
  \item\label{p-shortcat:l-reqmin} If $a \in J$, then $\min\sL(a) \le N$.
  \end{enumerate}
  Then $\sc_\sd(H) \le N+1$ for any distance $\sd$ on $H$ and $\max \Delta(H) \le N-1$.
  If, in addition, $J \cap H \setminus H^\times \ne \emptyset$, then $\rho(H) = \infty$.
\end{proposition}

\begin{proof}
  Replacing $J$ by $J \cap H\setminus H^\times$, we may without restriction assume $J \subset H \setminus H^\times$, since $\card{\sZ^*(a)}=1$ for any $a \in H^\times$.

  We first show $\sc_\sd(H) \le N+1$.
  If $a \in H \setminus J$, then $\sc_\sd(a)=0$ is immediate since $a$ has only one factorization.

  We now need to consider $a \in J$. We first prove the following auxiliary claim.
  \begin{claim}
  If $a \in J$ and $z \in \sZ^*(a)$, then there exist rigid factorizations $z=z_0$, $z_1$, $\ldots\,$,~$z_n \in \sZ^*(a)$ with $\sd(z_{i-1},z_i) \le N+1$ for all $i \in [1,n]$ and $\length{z_n} \le N$.
  \end{claim}

  Suppose $z \in \sZ^*(a)$ and $\length{z} > N$.
  Then $z=\rf{u_1,\cdots,u_k}$ with atoms $u_i \in \cA(H)$ and $k > N$.
  By \ref*{p-shortcat:reqsub} there exist $l \in [1,N+1]$ and $i \in [1,k-l+1]$ such that $u_i u_{i+1}\cdots u_{i+l-1} \in J$.
  Since $J$ is an ideal, we may take a longer subproduct if necessary, and assume without restriction $l=N+1$.
  By assumption \ref*{p-shortcat:l-reqmin}, the product $u_i\cdots u_{i+N}$ has a factorization of length at most $N$, that is, there exist $m \in [2,N]$ and $v_1$, $\ldots\,$,~$v_m \in \cA(H)$ such that $u_i\cdots u_{i+N} = v_1\cdots v_m$.
  Let
  \[
    z'=\rf{u_1,\cdots,u_{i-1},v_1,\cdots, v_m, u_{i+N+1},\cdots, u_{k}} \in \sZ^*(a).
  \]
  Then $z'$ is a factorization of $a$ with $\sd(z,z') \le N+1$ and $\length{z'} < \length{z}$.
  Iteration proves the claim.

  If $z_1$,~$z_2 \in \sZ^*(a)$, then there exist $z_1'$,~$z_2' \in \sZ^*(a)$ with $\length{z_1'} \le N$, $\length{z_2'} \le N$ and such that there are $(N+1)$-chains between $z_i$ and $z_i'$ for $i \in \{1,2\}$.
  Since $\sd(z_1',z_2') \le \max\{\length{z_1'},\length{z_2'}\} \le N$, these two chains can be concatenated to an $(N+1)$-chain between $z_1$ and $z_2$.
  Thus $\sc_\sd(a) \le N+1$.

  We now show $\max \Delta(H) \le n-1$.
  If $a \in H \setminus J$, then $\Delta(a) = \emptyset$.
  Suppose $a \in J$ and $z \in \sZ^*(H)$.
  If $\length{z} > N$ we may argue as before and replace $N+1$ atoms in $z$ by $m \in [2,N]$ new ones.
  For the resulting factorization $z'$, we have $\length{z} - (N-1) \le \length{z'} < \length{z}$.
  It follows that $\max \Delta(a) \le N-1$.

  Finally, we show $\rho(H) = \infty$.
  Let $a \in J \cap H \setminus H^\times$.
  Then $\min\sL(a^k) \le N$ while $\max\sL(a^k) \ge k$ for any $k \in \bN$.
  It follows that $\rho(a^k) \ge k/N$ for any $k \in \bN$, and so $\rho(H) \ge \infty$.
\end{proof}

\begin{remark}
  In an atomic monoid $H$, condition~\labelcref*{p-shortcat:reqsub} with the stronger requirement $N=0$ is equivalent to $J$ being completely prime.

  For commutative monoids, this requirement may be compared to the $\omega$-invariant of $J$.
  Indeed, if $H$ is commutative, we may replace \labelcref*{p-shortcat:reqsub} by the (weaker but similar) assumption $\omega(J) \le N+1$.
\end{remark}

We conclude this manuscript with a proof of \cref{t-main}.

\begin{proof}[Proof of \cref{t-main}]
  We first show that the natural inclusion homomorphism $R^\bullet \hookrightarrow \widehat R^\bullet$ satisfies the conditions of \cref{p-strongtransfer}.
  Recall that there is a bijection
  \begin{align*}
  \{\, \text{right $R$-ideals} \,\} &\ \leftrightarrow\  \{\, \text{right $\widehat R$-ideals} \,\} \\
  I        &\ \mapsto\  I \widehat R\\
  J \cap R &\ \mapsfrom\  J
  \end{align*}
  Moreover, if $I$,~$J$ are right $R$-ideals, then $I \cong J$ as $R$-modules if and only if $I\widehat R^\bullet \cong I \widehat R^\bullet$ as $\widehat R^\bullet$-modules.
  (See \cite[Corollary (30.10) and Proposition (30.17)]{Curtis-Reiner81}.)

  Let $a \in \widehat R^\bullet$.
  Then $I = a\widehat R \cap R$ is a right $R$-ideal.
  Since $I \widehat R = a\widehat R \cong \widehat R$, we conclude $I \cong R$.
  Thus $I$ is a principal right $R$-ideal and hence there exists $a_0 \in R^\bullet$ such that $I = a_0 R$.
  It follows that $a\widehat R^\bullet =a_0\widehat R^\bullet$ and thus there exists $\varepsilon \in \widehat R^\times$ such that $a=a_0 \varepsilon$.
  Now let $a$,~$b \in R^\bullet$ and $d \in \widehat R^\bullet$ such that $ad=b$.
  Then $d=a^{-1}b \in A^\times \cap \widehat R^\bullet = R^\bullet$.

  In the case that $\widehat A$ is a division ring, all claims of \cref{t-main} have been shown in \cref{t-div}.
  We may from now on assume $\widehat A \cong M_2(\widehat K)$.
  For $\widehat R^\bullet$ we have established in \cref{t-lengths,c-m2-minlen} that there exists an $N \in \bN_0$ (with $N=n+5$ where $n$ is the level of $R$ if $R$ is an Eichler order), such that:
  \begin{itemize}
  \item $\card{\sZ^*(a)} = 1$ for all $a \in \widehat{R}^\bullet \setminus J(\widehat R)$.
  \item $\min \sL(a) \le N$ for all $a \in J(\widehat R)$.
  \end{itemize}
  These properties carry over to $R$ by \cref{p-strongtransfer} (note $\widehat{J(R)} = J(\widehat R)$ and $J(\widehat R) \cap R = J(R)$).
  We now show that since $R/J(R) \cong \widehat R / J(\widehat R)$ is either a field or isomorphic to $\res \oplus \res$ with $\res =D/\pi D$, condition \ref{p-shortcat:reqsub} of \cref{p-shortcat} is also satisfied for $N \ge 1$. 
  
  First note that $u\in R^\times$ if and only if $u+J(R)$ is a unit of $R/J(R)$. Indeed, the canonical surjection $R\rightarrow R/J(R)$ is a ring homomorphism and if $uv\equiv 1\mod J(R)$, then $1-uv\in J(R)$.
  Hence $uv=1-(1-uv)\in R^\times$ and this implies that $u\in R^\times$.

We must now show the following: For $u_1 \cdots u_k \in J(R)$ with each $u_i\in \cA(R)$, there exists $i \in [1,k]$ such that either $u_i \in J(R)$ or $u_i u_{i+1} \in J(R)$. In the case of interest, $R/J(R) \cong \mathbb k \oplus \mathbb k$. Let $\overline{u_i}$ denote the image of $u_i$ in $\mathbb k \oplus \mathbb k$. Since each $u_i$ is a non-unit, at least one of the coordinates of each $\overline{u_i}$ is zero. If $\overline{u_i}=(0,0)$ for some $i$ we are done. If not, note that
since $u_1 \cdots u_k \in J(R)$ it must be the case that $\overline{u_1} \cdots \overline{u_k} = (0,0)$. This means that there exists at least one
$\overline{u_m}$ with the first coordinate zero, and at least one $\overline{u_n}$ with second coordinate zero. Since each $\overline{u_i}$ must have some coordinate equal to zero, we may assume that either $n=m+1$ or $m=n+1$. This proves that condition \ref{p-shortcat:reqsub} of \cref{p-shortcat} is satisfied.

Thus \cref{p-shortcat} implies claims \labelcref{t-main:rho,t-main:delta} of the theorem.
  Since $\Delta(R^\bullet)$ is finite and $\rho_k(R^\bullet)=\infty$ for $k \ge 2$, the Structure Theorem for Unions of Sets of Lengths holds by \cite[Theorem 2.6]{Geroldinger16}.
  If, in addition, $R$ is an Eichler order, then $\min\Delta(R^\bullet) =1$ by \cref{l-mindelta}.
\end{proof}

\hypersetup{ocgcolorlinks=false}
\bibliographystyle{hyperalphaabbr}
\bibliography{BaethSmertnigReferences}

\end{document}